\documentclass[12pt]{amsart}
\usepackage{amssymb}
\usepackage[all]{xy}
\usepackage{lscape}

\theoremstyle{plain} %% This is the default, anyway

\newtheorem{theorem}[equation]{Theorem}
\newtheorem{corollary}[equation]{Corollary}
\newtheorem{lemma}[equation]{Lemma}
\newtheorem{proposition}[equation]{Proposition}

\newtheorem{conjecture}[equation]{Conjecture}

\theoremstyle{definition}
\newtheorem{definition}[equation]{Definition}
\newtheorem{numerotation}[equation]{}

\theoremstyle{remark}
\newtheorem{remark}[equation]{Remark}
\newtheorem{example}[equation]{Example}

\newtheorem{propriete}[equation]{Property}

\newcommand{\Character}{\delta}
\def\build#1_#2^#3{\mathrel{\mathop{\kern0pt#1}\limits_{#2}^{#3}}}
\begin{document}
\begin{abstract}
Let $H$ be a Hopf algebra with a modular pair in involution $(\Character,1)$.
Let $A$ be a (module) algebra over $H$
equipped with a non-degenerated $\Character$-invariant $1$-trace $\tau$.
We show that Connes-Moscovici characteristic map
$\varphi_\tau:HC^*_{(\Character,1)}(H)\rightarrow HC^*_\lambda(A)$
is a morphism of graded Lie algebras.
We also have a morphism $\Phi$ of Batalin-Vilkovisky algebras
from the cotorsion product of $H$, $\text{Cotor}_H^*({\Bbbk},{\Bbbk})$,
to the Hochschild cohomology of $A$, $HH^*(A,A)$.
Let $K$ be both a Hopf algebra and a symmetric Frobenius algebra.
Suppose that the square of its antipode is an inner automorphism by a group-like element.
Then this morphism of Batalin-Vilkovisky algebras
$\Phi:\text{Cotor}_{K^\vee}^*(\mathbb{F},\mathbb{F})\cong \text{Ext}_{K}(\mathbb{F},\mathbb{F})
\hookrightarrow HH^*(K,K)$
is injective.
\end{abstract}
\title{\bf Connes-Moscovici characteristic map is a Lie algebra morphism}
\author{Luc Menichi}
\address{UMR 6093 associ\'ee au CNRS\\
Universit\'e d'Angers, Facult\'e des Sciences\\
2 Boulevard Lavoisier\\49045 Angers, FRANCE}
\email{firstname.lastname at univ-angers.fr}
%\thanks{Research supported by the University of Toronto (NSERC grants RGPIN 8047-98 and OGP000 7885)}\subjclass{55P35, 16E40, 55P62, 57T30, 55U10}
%\subjclass{16E40, 16E45, 55P35, 57P10}
\maketitle
\begin{center}
%\textit{Dedicated to Jean-Claude Thomas, on the occasion of his
%60th birthday}
\end{center}
%\cite{mastpevschauwith:cohomologyhopf},~\cite[Conjecture 2.18]{Etingof-Ostrik:finitetensorcat}
\section{Introduction}
Let ${\Bbbk}$ be any commutative ring and $\mathbb{F}$ be any field.
It is well known that the Hochschild cohomology of an algebra $A$, $HH^*(A,A)$,
is a Gerstenhaber algebra.
It is also well known that the homology of a double pointed loop space, $H_*(\Omega^2 X)$,
is also a Gerstenhaber algebra~\cite{Cohen-Lada-May:homiterloopspaces}.
Let $H$ be a bialgebra. It is not well known (See~\cite{Kadeishvili:cobarbialgebra}
for a recent paper rediscovering it) that the Cotorsion product of $H$,
 $\text{Cotor}^{*}_H(\Bbbk,\Bbbk)$
has a Gerstenhaber algebra structure
(this results from~\cite[p. 65]{Gerstenhaber-Schack:algbqgad}).
But it should. Indeed, by Adams cobar equivalence,
there is an isomorphism 
$\text{Cotor}^{*}_{S_*(\Omega X)}(\Bbbk,\Bbbk)\cong H_*(\Omega^2 X)$
between the two Gerstenhaber algebras (See the proof of Corollary~\ref{lacet double sous algebre
Gerstenhaber de Hochschild} for details).

The first goal of this paper is to study (Section~\ref{algebre de Gerstenhaber sur le produit exterieur})
this Gerstenhaber algebra
 $\text{Cotor}^{*}_H(\Bbbk,\Bbbk)$.
In particular, generalizing a result of Farinati and
Solotar~\cite{Farinati-Solotar:GstrcohHopf}, we show (Theorem~\ref{ext sous algebre de Gerstenhaber de Hochschild})
that the exterior product
$\text{Ext}^{*}_H(\Bbbk,\Bbbk)$ is a sub Gerstenhaber algebra
of the Hochschild cohomology of $H$, $HH^*(H,H)$.

In Section~\ref{algebres de Batalin-Vilkovisky}, we turn our attention to a particular
case of Gerstenhaber algebras: the Batalin-Vilkovisky algebras.
In~\cite{MenichiL:BValgaccoHa}, we introduced the notion of cyclic operad with multiplication
(Definition~\ref{operade cyclique avec multiplication}) and we showed
(Theorem~\ref{Theoreme BV algebre})
that every cyclic operad with multiplication $\mathcal{O}$ gives a cocyclic module
such that

-the homology of the associated cochain complex $H(\mathcal{C}^*(\mathcal{O}))$
is a Batalin-Vilkovisky algebra and

-the negative cyclic cohomology of $\mathcal{C}^*(\mathcal{O})$, $HC_-^*(\mathcal{O})$,
has a Lie bracket of degre $-2$.

Let $M$ be a simply-connected closed manifold.
In~\cite{Chas-Sullivan:stringtop}, Chas and Sullivan showed that $\mathbb{H}_*(LM)$,
the free loop space homology of $M$, is a Batalin-Vilkovisky algebra
and that the $S^1$-equivariant homology $H_*^{S^1}(LM)$ has a Lie bracket.
The singular cochains of $M$, $S^*(M)$ is a (derived) symmetric Frobenius algebra.
Motivated by Chas-Sullivan string topology, in~\cite{MenichiL:BValgaccoHa},
as first application of Theorem~\ref{Theoreme BV algebre}, we obtained that
the Hochschild cohomology of a symmetric Frobenius algebra $A$, $HH^*(A,A)$,
is a Batalin-Vilkovisky algebra and that the negative cyclic cohomology of $A$,
$HC_-^*(A)$ has a Lie bracket of degre $-2$.
It is expected that there is an isomorphism of Batalin-Vilkovisky algebras
$HH^*(S^*(M),S^*(M))\cong  \mathbb{H}_*(LM)$ and an isomorphism of Lie algebras
$HC_-^*(S^*(M))\cong H_*^{S^1}(LM)$.

In~\cite{Getzler:BVAlg}, Getzler showed that the Gerstenhaber algebra
$H_*(\Omega^2 X)$ is in fact a Batalin-Vilkovisky algebra.
Therefore as second application of Theorem~\ref{Theoreme BV algebre},
in~\cite{MenichiL:BValgaccoHa}, we showed that the Cotorsion product of a Hopf algebra
$H$ with an involutive antipode or more generally with a modular pair in involution
$(\Character,1)$, $\text{Cotor}^{*}_H(\Bbbk,\Bbbk)$, is a Batalin-Vilkovisky algebra.
In this paper, we give the dual result (Theorem~\ref{ext BV algebre})
which we believe is far more clear: Let $K$ be a Hopf algebra such that the square of
its antipode is an inner automorphism by a group-like element. Then
$\text{Ext}^{*}_K(\Bbbk,\Bbbk)$ is a Batalin-Vilkovisky algebra.

In~\cite{MenichiL:BValgaccoHa}, we also had that the negative cyclic cohomology of $H$,
$HC_{-(\Character,1)}^*(H)$ has a Lie bracket of degre $-2$.
But Connes and Moscovici never use negative cyclic cohomology: they use the
(ordinary) cyclic cohomology. Therefore, in this paper, we show
(Corollary~\ref{crochet sur Hopf cyclic cohomologie}) that
Connes-Moscovici (ordinary) cyclic cohomology of $H$,
$HC^*_{(\Character,1)}(H)$, has also a Lie bracket (of degree $-1$ this time)
and we show (Theorem~\ref{characteristic de Connes-Moscovici morphisme de Lie}
and its variant Theorem~\ref{characteristic de Khalkhali-Rangipour morphisme de Lie}) that Connes-Moscovici characteristic map
$\chi_\tau:HC_{(\Character,1)}^*(H)\rightarrow HC^*_\lambda(A)$
is compatible with the Lie brackets of degre $-1$.
Here $A$ is a symmetric Frobenius algebra equipped an action of the Hopf algebra $H$
compatible with the product and the trace.

In fact, we show that Connes-Moscovici characteristic map is induced by a morphism
of cyclic operads with multiplication from the cobar construction of $H$, $\Omega H$,
to the Hochschild cochain complex of $A$, $\mathcal{C}^*(A,A)$.
And we show that the (ordinary) cyclic cohomology of every cyclic operad with
multiplication has naturally a Lie bracket of degre $-1$
(Theorem~\ref{crochet sur cohomologie cyclique lambda}).
As a consequence of Theorem~\ref{Theoreme BV algebre}, we also obtain
a morphism of Batalin-Vilkovisky algebras
$H^*(\Phi):\text{Cotor}^*_H({\Bbbk},{\Bbbk})\rightarrow HH^*(A,A)$(Theorem~\ref{characteristic de Connes-Moscovici morphisme de Lie}).

Note that this morphism $H^*(\Phi)$ should be the algebraic counterpart
of our very recent morphism of
Batalin-Vilkovisky algebras~\cite[Theorem 24]{Menichi:BVmorphismdoublefree}
$$
\text{Cotor}_{S_*(G)}({\Bbbk},{\Bbbk})\cong H_*(\Omega^2 BG)\rightarrow\mathbb{H}_*(LM)
\cong HH^*(S^*(M),S^*(M))$$
between the Batalin-Vilkovisky algebra on the homology of double loop space
given by by Getzler~\cite{Getzler:BVAlg}, and the Batalin-Vilkovisky algebra on the free
loop space homology of a manifold given by Chas and Sullivan. Here $G$ is a topological group acting on $M$.

In Section 8, we specialize to the case where the symmetric Frobenius algebra
$A$ is the Hopf algebra $H$ itself. And we show that the inclusion of Gerstenhaber
algebras $\text{Ext}^{*}_H(\mathbb{F},\mathbb{F})\hookrightarrow HH^*(H,H)$,
given by Theorem~\ref{ext sous algebre de Gerstenhaber de Hochschild},
is often an inclusion of Batalin-Vilkovisky algebras
(Theorem~\ref{inclusion BV ext hochschild}).

In this last section, we compute the Batalin-Vilkovisky algebra structure on $\text{Cotor}_{H}({\Bbbk},{\Bbbk})$ introduced in~\cite[Theorem 1.1]{MenichiL:BValgaccoHa} and recalled
in Corollary~\ref{Cobar algebre de Hopf BV} when $H$ is the universal envelopping algebra
of a Lie algebra over a field of characteristic $0$.

{\em Acknowledgment:}
We wish to thank Yves F\'elix and Jean-Claude Thomas for interesting comments
on Corollary~\ref{lacet double sous algebre Gerstenhaber de Hochschild}.

\section{Hochschild complex and (Co)bar construction}
We work over an arbitrary commutative ring ${\Bbbk}$,
except for Conjectures~\ref{conjecture crochet trivial pour bialgebre tresse}
-\ref{conjecture ext 3-algebre} in 
Section~\ref{algebre de Gerstenhaber sur le produit exterieur},
for Proposition~\ref{dualite Connes-Moscovici Khalkhali-Rangipour}
to Corollary~\ref{dualite operad cyclic bar cobar}
(almost all Section~\ref{section ext BV-algebre})
and for all Section~\ref{algebre de Hopf symmetrique},
where we use  an arbitrary field $\mathbb{F}$ as coefficient.

Let $A$ be an algebra and $M$ be a $A$-bimodule.
The {\it Hochschild chain complex} $\mathcal{C}_*(A,M)$ is the chain complex
$\mathcal{C}_n(A,M)=M\otimes A^n$ with differential $d:\mathcal{C}_n(A,M)\rightarrow\mathcal{C}_{n-1}(A,M)$
given by
\begin{multline*}
d(m\otimes a_1\otimes\dots\otimes a_n)=ma_1\otimes a_2\otimes\dots\otimes a_n\\
+\sum_{i=1}^{n-1}(-1)^{i}m\otimes a_1\otimes\dots\otimes a_ia_{i+1}\otimes\dots\otimes a_n
+(-1)^n a_nm\otimes a_1\otimes\dots\otimes a_{n-1}.
\end{multline*}
By definition, the {\it Hochschild homology of $A$ with coefficients in $M$}, 
$HH_*(A,M)$ is the homology of $\mathcal{C}_*(A,M)$.
The {\it Hochschild cochain complex} $\mathcal{C}^*(A,M)$ is the cochain complex
$\mathcal{C}^n(A,M)=\text{Hom}(A^n,M)$ with differential $d:\mathcal{C}^n(A,M)\rightarrow\mathcal{C}^{n+1}(A,M)$
given by
\begin{multline*}
d(f)(a_0\otimes\dots\otimes a_n)=a_0f(a_1\otimes\dots\otimes a_n)\\
+\sum_{i=1}^{n}(-1)^{i}f(a_0\otimes\dots\otimes a_{i-1}a_{i}\otimes\dots\otimes a_n)
+(-1)^{n+1} f(a_0\otimes\dots\otimes a_{n-1}) a_n.
\end{multline*}
By definition, the {\it Hochschild cohomology of $A$ with coefficients in $M$}, 
$HH^*(A,M)$ is the homology of $\mathcal{C}^*(A,M)$.
Suppose that $A$ has an augmentation $\varepsilon:A\rightarrow {\Bbbk}$. Then ${\Bbbk}$ is a $A$-bimodule.
The {\it (reduced) Bar construction} $B(A)$ is just then the 
Hochschild chain complex $\mathcal{C}_*(A,{\Bbbk})$
and $\text{Ext}^*_A({\Bbbk},{\Bbbk})=HH^*(A,{\Bbbk})$. 

Dually, let $C$ be a coalgebra with diagonal $\Delta_C:C\rightarrow C\otimes C$.
Let $N$ be a $C$-bicomodule with left $C$-coaction
$\Delta_{N}^{l}:N\rightarrow C\otimes N$. and right $C$-coaction
$\Delta_{N}^{r}:N\rightarrow N\otimes C$.
The Hochschild cochain complex $\mathcal{C}_{coalg}^*(C,N)$
(~\cite[p. 57]{Gerstenhaber-Schack:algbqgad}
or~\cite[30.3]{Brzezinski-Wisbauer:corings}) is the cochain complex
$\mathcal{C}^n(C,N)=\text{Hom}(N,C^n)$ with differential $d:\mathcal{C}^n(C,N)\rightarrow\mathcal{C}^{n+1}(C,N)$
given by
$$d(\varphi)=(C\otimes\varphi)\circ \Delta^l_N+\sum_{i=1}^n(-1)^i
(C^{\otimes i-1}\otimes\Delta_C\otimes C^{\otimes n-i})\circ\varphi
+(-1)^{n+1}(\varphi\otimes C)\circ \Delta^r_N.
$$
The {\it Hochschild coalgebra cohomology} $HH^*_{coalg}(C,N)$ is the homology
of  $\mathcal{C}_{coalg}^*(C,N)$.
Suppose that $C$ has a coaugmentation $\eta:{\Bbbk}\rightarrow C$.
Then ${\Bbbk}$ is a $C$-bicomodule.
The {\it (reduced) cobar construction}
$\Omega(C)$~\cite[p. 432]{Kassel:quantumgrps} is just
 $\mathcal{C}^*_{coalg}(C,{\Bbbk})$
and $\text{Cotor}^*_C({\Bbbk},{\Bbbk})=HH^*_{coalg}(C,{\Bbbk})$. 
\section{Operads with multiplication}
A {\it Gerstenhaber algebra} is a commutative
graded algebra $A=\{A^i\}_{i\in
\mathbb{Z}}$
equipped with a bracket of degre $-1$
$$
\{-,-\}:A^i \otimes A^j \to A^{i+j-1} \,, \quad x\otimes y \mapsto \{x,y\}
$$
such that the product and the Lie bracket satisfy the Poisson rule:
for any  $c  \in A^k$  
the adjunction map $\{-,c\}:A^i \to A^{i+k-1} \,, \quad
a\mapsto \{a,c\}$ is  a
$(k-1)$-derivation: ie. for $a$, $b$, $c\in A$,
$\{ ab,c\} = \{ a,c\} b + (-1)^{\vert a\vert (\vert c\vert -1)} 
a \{ b,c\}$.

In this paper, every Gerstenhaber algebra comes from a (linear)
operad with multiplication using the following general theorem:
\begin{theorem}\cite{Gerstenhaber-Schack:algbqgad,Gerstenhaber-Voronov:hgamso,McClure-Smith:deligneconj}\label{Theoreme Gerstenhaber}
\noindent a) Each operad with multiplication $O$ is a cosimplicial
module (See~\ref{definition du module cosimplicial}).
Denote by $\mathcal{C}^{*}(O)$ the associated cochain complex. 

\noindent b) Its homology $H(\mathcal{C}^{*}(O))$
is a Gerstenhaber algebra.
\end{theorem}
Let us first recall what is a (linear) operad and a (linear) operad with
multiplication.

In this paper, {\it operad} means non-$\Sigma$-operad in the category of
$\Bbbk$-modules. That is:
a sequence of modules $\{O(n)\}_{n\geq 0}$, an identity element $id\in O(1)$
and structure maps

$\gamma:O(n)\otimes O(i_1)\otimes\dots\otimes O(i_n)\rightarrow O(i_1+\dots+i_n)$

$f\otimes g_1\otimes\dots\otimes g_n\mapsto \gamma(f;g_1,\dots,g_n)$

\noindent
satisfying associativity and unit~\cite{Markl-Shnider-Stasheff:opeatp}.

Hereafter we use mainly the composition operations
$\circ_i:O(m)\otimes O(n)\rightarrow O(m+n-1)$,
$f\otimes g\mapsto f\circ_i g$
defined for $m\in\mathbb{N}^{*}$, $n\in\mathbb{N}$ and $1\leq i\leq m$
by $f\circ_i g:=\gamma(f;id,\dots,g,id,\dots,id)$ where $g$ is the $i$-th element
after the semicolon.
\begin{example}\label{operade endomorphisme}
Let $(\mathcal{C},\otimes,\Bbbk)$ be a monoidal category.
Suppose that $\mathcal{C}$ is enriched over the category of
$\Bbbk$-modules~\cite[I.8]{MacLane:catwm}
and that
$$\otimes: \text{Hom}_\mathcal{C}(V_1,W_1)\times \text{Hom}_\mathcal{C}(V_2,W_2)\rightarrow \text{Hom}_\mathcal{C}(V_1\otimes V_2,W_1\otimes W_2),$$
mapping $(g_1,g_2)$ to $g_1\otimes g_2$, is ${\Bbbk}$-bilinear
(we say that $\mathcal{C}$ is a {\it ${\Bbbk}$-linear monoidal category}).
Let $V$ be a object of $\mathcal{C}$. The {\it endomorphism operad} of $V$
in $\mathcal{C}$~\cite[p. 43]{Markl-Shnider-Stasheff:opeatp} is the operad
$\mathcal{E}nd_\mathcal{C}(V)$ defined by $$\mathcal{E}nd_\mathcal{C}(V)(n):=\text{Hom}_\mathcal{C}(V^{\otimes n},V).$$
The structure maps $\gamma$
$$\text{Hom}_\mathcal{C}(V^{\otimes n},V)\otimes \text{Hom}_\mathcal{C}(V^{\otimes i_1},V)\otimes\dots\otimes \text{Hom}_\mathcal{C}(V^{\otimes i_n},V)\rightarrow \text{Hom}_\mathcal{C}(V^{\otimes i_1+\dots+i_n},V)$$
are given by
$\gamma(f;g_1,\dots,g_n)=f\circ (g_1\otimes\dots\otimes g_n)$.
The identity element of $\mathcal{E}nd_\mathcal{C}(V)$ is the identity map $id_V:V\rightarrow V$.
\end{example}
\begin{example}
The {\it coendomorphism operad} of $V$
in $\mathcal{C}$, denoted
$\mathcal{C}o\mathcal{E}nd_\mathcal{C}(V)$, is by definition
the endomorphism operad of $V$ in the opposite category $\mathcal{C}^{op}$,
$\mathcal{E}nd_{\mathcal{C}^{op}}(V)$.
Explicitly~\cite[p. 43-4]{Markl-Shnider-Stasheff:opeatp} $\mathcal{C}o\mathcal{E}nd_\mathcal{C}(V)$
is the operad given by
$$\mathcal{C}o\mathcal{E}nd_\mathcal{C}(V)(n):=\text{Hom}_\mathcal{C}(V,V^{\otimes n}).$$
The structure maps $\gamma$
$$\text{Hom}_\mathcal{C}(V,V^{\otimes n})\otimes \text{Hom}_\mathcal{C}(V,V^{\otimes i_1})\otimes\dots\otimes \text{Hom}_\mathcal{C}(V,V^{\otimes i_n})\rightarrow \text{Hom}_\mathcal{C}(V,V^{\otimes i_1+\dots+i_n})$$
are given by
$\gamma(f;g_1,\dots,g_n)=(g_1\otimes\dots\otimes g_n)\circ f$.
The identity element of $\mathcal{E}nd_\mathcal{C}(V)$ is again $id$.
\end{example}
\begin{definition}\label{definition operade avec multiplication}
An {\it operad with multiplication} is an operad equipped with an element $\mu\in O(2)$ called the multiplication and an element $e\in O(0)$
such that
$\mu\circ_1\mu=\mu\circ_2\mu$
and $\mu\circ_1 e=id=\mu\circ_2 e.$
\end{definition}
Let $\underline{Ass}$ be the (non-$\Sigma$) associative
operad~\cite{Markl-Shnider-Stasheff:opeatp}: $\underline{Ass}(n):={\Bbbk}$.
An operad $O$ is an operad with multiplication if and only if
$O$ is equipped with a morphism of operads $\underline{Ass}\rightarrow O$.
\begin{proof}[Sketch of proof of~\ref{Theoreme Gerstenhaber}]

a) The coface maps $\delta_{i}:O(n)\rightarrow O(n+1)$ and codegeneracy maps
$\sigma_{i}:O(n)\rightarrow O(n-1)$
are defined \cite{McClure-Smith:deligneconj} by
\begin{equation}\label{definition du module cosimplicial}
\delta_{0}f=\mu\circ_2 f,\;
\delta_{i}f=f\circ_i\mu,\;
\delta_{n+1}f=\mu\circ_1 f,\;
\sigma_{i-1}f=f\circ_{i}e\text{ for }1\leq i\leq n.
\end{equation}
b) The associated cochain complex $\mathcal{C}^{*}(O)$ is the cochain
complex whose differential $d$ is given by $$d:=\sum_{i=0}^{n+1}(-1)^{i}\delta_{i}:O(n)\rightarrow O(n+1).$$
The linear maps $\cup:O(m)\otimes O(n)\rightarrow O(m+n)$ defined by
\begin{equation}\label{definition du cup}
f\cup g:=(\mu\circ_1 f)\circ_{m+1}g=(\mu\circ_2 g)\circ_{1}f
\end{equation}
gives $\mathcal{C}^{*}(O)$ a structure of differential graded algebra.
The linear maps of degree $-1$
$$\overline{\circ},\{-,-\}:O(m)\otimes O(n)\rightarrow O(m+n-1)$$
are defined by
\begin{equation}\label{definition du composition}
f\overline{\circ}g:=(-1)^{(m-1)(n-1)}\sum_{i=1}^{m}
(-1)^{(n-1)(i-1)}f\circ_i g
\end{equation}
and $$\{f,g\}:=f\overline{\circ}g-(-1)^{(m-1)(n-1)}g\overline{\circ}f.$$
The bracket $\{-,-\}$ defines a structure of differential graded Lie algebra
of degree $-1$ on $\mathcal{C}^{*}(O)$.
After passing to cohomology, the cup product $\cup$ and the bracket $\{-,-\}$
satisfy the Poisson rule.
\end{proof}
\begin{remark}\label{operation sur Steenrod sur cohomologie de Hochschild}
As pointed by Turchin in~\cite{turchin:dyeylashof}, the Gerstenhaber algebra
$H(\mathcal{C}^*(\mathcal{O})$ has Dyer-Lashof operations.
In particular~\cite[p. 63]{Gerstenhaber-Schack:algbqgad}, if $n$ is even
or if $2=0$ in ${\Bbbk}$, a Steenrod or Dyer-lashof (non additive) operation
$Sq^{n-1}:H^n(\mathcal{C}^*(\mathcal{O})\rightarrow H^{2n-1}(\mathcal{C}^*(\mathcal{O})$ is defined by $Sq^{n-1}(f)=f\overline{\circ}f$ for $f\in \mathcal{O}(n)$.
\end{remark}
\begin{remark}\label{crochet en degre un}
Let $\mathcal{O}$ be an operad.
Then $\mathcal{O}(1)$ equipped with
$\circ_1:\mathcal{O}(1)\otimes\mathcal{O}(1)\rightarrow\mathcal{O}(1)$
and $id:{\Bbbk}\rightarrow \mathcal{O}(1)$ is an algebra.
By~(\ref{definition du composition}), the Lie algebra $\mathcal{C}^{1}(O)$ 
is just  $\mathcal{O}(1)$ equipped with the Lie bracket given by
$
\{f,g\}:=f\circ_1 g-g\circ_1 f
$.
\end{remark}
\begin{example}
Let $A$ be a monoid in $\mathcal{C}$, i. e.
an object of  $\mathcal{C}$ equipped with an 
associative multiplication
$\mu:A\otimes A\rightarrow A$ and an unit $e:\Bbbk\rightarrow A$.
Then the endomorphism operad $\mathcal{E}nd_\mathcal{C}(A)$ of $A$ equipped
with
$\mu\in \text{Hom}_\mathcal{C}(A^{\otimes 2},A)=\mathcal{E}nd_\mathcal{C}(A)(2)$ and $e\in \text{Hom}_\mathcal{C}(A^{\otimes 0},A)=\mathcal{E}nd_\mathcal{C}(A)(0)$ is an operad with multiplication.
The associated cosimplicial module is the cosimplicial module
$\{\text{Hom}_\mathcal{C}(A^{\otimes n},A)\}_{n\in\mathbb{N}}$.
The coface maps $\delta_i:\text{Hom}_\mathcal{C}(A^{\otimes n},A)
\rightarrow\text{Hom}_\mathcal{C}(A^{\otimes n+1},A)$
and the codegeneracy map $\sigma_i:\text{Hom}_\mathcal{C}(A^{\otimes n},A)
\rightarrow\text{Hom}_\mathcal{C}(A^{\otimes n-1},A)$ are given
by~\cite[(2.5)]{MenichiL:BValgaccoHa}
\begin{equation}\label{cosimplicial associe a endomorphism}
\delta_{0}f=\mu\circ (id\otimes f),\;
\delta_{i}f=f\circ (id^{\otimes i-1}\otimes \mu\otimes id^{\otimes n-i}),\;
\delta_{n+1}f=\mu\circ (f\otimes id),\;
\end{equation}
and $\sigma_{i-1}f=f\circ (id^{\otimes i-1}\otimes e\otimes id^{\otimes n-i})$
for $1\leq i\leq n$.
\end{example}
If $\mathcal{C}$ is the category of ${\Bbbk}$-modules,
$A$ is an algebra and the cochain complex $\mathcal{C}^{*}(\mathcal{E}nd_\mathcal{C}(A))$ associated to this cosimplicial module is
the Hochschild cochain complex of $A$, denoted $\mathcal{C}^{*}(A,A)$.
This is why Turchin in his work on knots~\cite{turchin:homologyknots,turchin:bialgebrachord} always call the cochain
complex associated to a linear operad with multiplication,
the Hochschild cochain complex of the operad with multiplication.
\begin{propriete}\label{functor monoidal induit morphism d'operades}
Let $\mathcal{C}$ and $\mathcal{D}$ be two ${\Bbbk}$-linear monoidal
categories.
Let $F:\mathcal{C}\rightarrow \mathcal{D}$ be a monoidal functor
(in the sense of~\cite[p. 255]{MacLane:catwm}).
Let $\psi:F(V)\otimes F(W)\rightarrow F(V\otimes W)$ be the associated
associative unital natural transformation.
Suppose that
$F:\text{Hom}_\mathcal{C}(V,W)\rightarrow \text{Hom}_\mathcal{D}(F(V),F(W))$
is ${\Bbbk}$-linear.
Let $A$ be a monoid in $\mathcal{C}$. Then $F(A)$ is a monoid in $\mathcal{D}$
and the map $\Gamma$ from $\mathcal{E}nd_\mathcal{C}(A)$ to $\mathcal{E}nd_\mathcal{D}(F(A))$,
mapping $f:A^{\otimes n}\rightarrow A$ to the composite $F(f)\circ \psi:F(A)^{\otimes n}\rightarrow F(A)$,
is a morphism of operads with multiplication.
\end{propriete}
\begin{example}
Dually, let $C$ be a comonoid in  $\mathcal{C}$, i. e. an object
of  $\mathcal{C}$ equipped with a coassociative diagonal
$\Delta:C\rightarrow C\otimes C$
and a counit $\varepsilon:C\twoheadrightarrow {\Bbbk}$.
Since $C$ is a monoid in $\mathcal{C}^{op}$,
the coendomorphism operad of $C$, 
$\mathcal{C}o\mathcal{E}nd_\mathcal{C}(C)$
equipped with 
$\Delta\in \text{Hom}_\mathcal{C}(C,C^{\otimes 2})=\mathcal{C}o\mathcal{E}nd_\mathcal{C}(C)(2)$ and $\varepsilon\in \text{Hom}_\mathcal{C}(C, C^{\otimes 0})=\mathcal{C}o\mathcal{E}nd_\mathcal{C}(C)(0)$ is also an operad with multiplication.
The associated cosimplicial module is the cosimplicial module
$\{\text{Hom}_\mathcal{C}(C,C^{\otimes n})\}_{n\in\mathbb{N}}$.
The coface maps $\delta_i:\text{Hom}_\mathcal{C}(C,C^{\otimes n})
\rightarrow\text{Hom}_\mathcal{C}(C,C^{\otimes n+1})$
and the codegeneracy map $\sigma_i:\text{Hom}_\mathcal{C}(C,C^{\otimes n})
\rightarrow\text{Hom}_\mathcal{C}(C,C^{\otimes n-1})$ are given
by
\begin{equation}\label{cosimplicial associe a coendomorphism}
\delta_{0}f=(id\otimes f)\circ\Delta,\;
\delta_{i}f=(id^{\otimes i-1}\otimes \Delta\otimes id^{\otimes n-i})\circ f,\;
\delta_{n+1}f=(f\otimes id)\circ \Delta,\;
\end{equation}
and $\sigma_{i-1}f=(id^{\otimes i-1}\otimes \varepsilon\otimes id^{\otimes n-i})\circ f$
for $1\leq i\leq n$.
\end{example}
If $\mathcal{C}$ is the category of ${\Bbbk}$-modules,
$C$ is a coalgebra and the cochain complex $\mathcal{C}^{*}(\mathcal{C}o\mathcal{E}nd_\mathcal{C}(C))$ associated to this cosimplicial module is
the Hochschild cochain complex of the coalgebra $C$, denoted $\mathcal{C}^{*}_{coalg}(C,C)$.

More generally, let $A$ be ${\Bbbk}$-algebra.
 Let $\mathcal{C}$ be the category of $A$-bimodules.
Let $C$ be a $A$-coring, i. e. a comonoid in $\mathcal{C}$
(\cite[4.2]{Kadison:Frobeniusextension} or~\cite[17.1]{Brzezinski-Wisbauer:corings}).
The cochain complex $\mathcal{C}^{*}(\mathcal{C}o\mathcal{E}nd_\mathcal{C}(C))$ associated to this cosimplicial module is
the Cartier cochain complex of $C$ with coefficients in $C$, denoted $C_{Ca}(C,C)$.
Therefore, without any calculations, we have obtained that
$C_{Ca}(C,C)$ is an operad with multiplication~\cite[30.8]{Brzezinski-Wisbauer:corings}.
This is again an example of our leitmotiv in this paper:

``Every operad with multiplication should be the endomorphism operad of a monoid in a appropriate monoidal category $\mathcal{C}$''.
\section{Gerstenhaber algebra structure on $\text{Ext}_A^*({\Bbbk},{\Bbbk})$}\label{algebre de Gerstenhaber sur le produit exterieur}

Let $C$ be a bialgebra.
The Cobar construction of $C$ is the cosimplicial module associated to a
specific linear operad with multiplication~\cite[p. 65]{Gerstenhaber-Schack:algbqgad}.
Therefore its cohomology $\text{Cotor}^{*}_C(\Bbbk,\Bbbk)$
has a Gerstenhaber algebra structure.
In the following, we show that this operad with multiplication
is just the endomorphism operad of a monoid in an appropriate monoidal
category and we show:
\begin{theorem}\label{cotor sous algebre de Gerstenhaber de Hochschild}
Let $C$ be a bialgebra.
Then $\text{Cotor}^{*}_C(\Bbbk,\Bbbk)$ is a sub Gerstenhaber algebra
of the Hochschild cohomology of the coalgebra $C$, $HH^*_{coalg}(C,C)$.
\end{theorem}
By Property~\ref{cotor en degre un egal primitif comme algebre de Lie},
this Lie bracket of degre $-1$ on the cotorsion product of a bialgebra is an extension
of the well-known Lie bracket on the primitive elements of a bialgebra.
Dually, we prove
\begin{theorem}\label{ext sous algebre de Gerstenhaber de Hochschild}
Let $A$ be a bialgebra.
Then $\text{Ext}^{*}_A(\Bbbk,\Bbbk)$ is a sub Gerstenhaber algebra
of the Hochschild cohomology of the algebra $A$, $HH^*(A,A)$.
\end{theorem}
When $A$ is a Hopf algebra, this theorem was proved by Farinati and
Solotar~\cite{Farinati-Solotar:GstrcohHopf}.
But as we would like to emphasize, antipodes are not needed for
the Gerstenhaber algebra structure. As we explain in
Theorem~\ref{ext BV algebre},
antipodes are needed only to have a Batalin-Vilkovisky algebra structure.

By Property~\ref{ext en degre un egal dual indecomposable comme algebre de Lie}, this inclusion of Gerstenhaber algebras is in degre $1$ the inclusion of the Lie
algebra of ``differentiations'' into the Lie algebra of derivations, well known in algebraic
groups.
%this Lie bracket of degre $-1$ on the exterior product of a bialgebra is an extension
%of the Lie bracket on the dual of the indecomposable elements of a bialgebra.

In Proposition~\ref{dualite entre inclusions Gerstenhaber}, we prove that
when the bialgebra $C$ is ${\Bbbk}$-free of finite type, Theorem~\ref{ext sous algebre de Gerstenhaber de Hochschild} is the dual of
Theorem~\ref{cotor sous algebre de Gerstenhaber de Hochschild}. This duality will
be later extended in Corollary~\ref{dualite operad cyclic bar cobar}.

In Conjectures~\ref{conjecture crochet trivial pour bialgebre tresse}
and~\ref{conjecture ext 3-algebre}, we explain that if the bialgebra $A$
is braided, the Lie bracket of degre $-1$ given by Theorem~\ref{ext sous algebre de Gerstenhaber de Hochschild} on $\text{Ext}^*_A(\mathbb{F},\mathbb{F})$ should vanish and be replaced by a Lie bracket of degre $-2$.
This is related to a conjecture of Kontsevich.

In Corollary~\ref{lacet double sous algebre Gerstenhaber de Hochschild},
we explain that the homology of a double loop space $H_*(\Omega^2 X)$
is always a sub Gerstenhaber algebra of Hochschild cohomology if
$X$ is $2$-connected.

In Corollary~\ref{cohomology d'un espace sous algebre Gerstenhaber de Hochschild},
we show that the cohomology algebra of any path-connected topological space is also
a sub Gerstenhaber algebra of Hochschild cohomology.
\begin{proof}[Proof of Theorem~\ref{cotor sous algebre de Gerstenhaber de Hochschild}]
The category of left $C$-modules, $C$-mod, is a monoidal category.
Let $M$ be a comonoid in this monoidal category, i. e. $M$ is a $C$-module
coalgebra~\cite[Definition IX.2.1]{Kassel:quantumgrps}.
The coendomorphism operad associated to $M$ is the operad
$\{\text{Hom}_{C-mod}(M,M^{\otimes n})\}_{n\in\mathbb{N}}$
with multiplication $\Delta:M\rightarrow M\otimes M\in \text{Hom}_{C-mod}(M,M^{\otimes 2})$
and  $\varepsilon:M\rightarrow{\Bbbk}\in \text{Hom}_{C-mod}(M,M^{\otimes 0})$.
The inclusion maps $i_C:\text{Hom}_{C-mod}(M,M^{\otimes n})
\hookrightarrow \text{Hom}_{{\Bbbk}-mod}(M,M^{\otimes n})
$
defines obviously a morphism of linear operads with multiplication.

The underlying coalgebra $C$ is an example of $C$-module coalgebra.
Therefore we can take in particular $M=C$.
The linear morphism
$ev:\text{Hom}_{C-mod}(C,C^{\otimes n})\buildrel{\cong}\over\rightarrow C^{\otimes n}$,
mapping $f:C\rightarrow C^{\otimes n}$ to $f(1)$ is an isomorphism.
The inverse is the linear map
$ext:C^{\otimes n}\buildrel{\cong}\over\rightarrow \text{Hom}_{C-mod}(C,C^{\otimes n})$,
mapping $c_1\otimes\dots\otimes c_n$ to $f:C\rightarrow C^{\otimes n}$
defined by $f(c)=c^{(1)}c_1\otimes\dots\otimes c^{(n)}c_n$.
Here we have denoted by $c^{(1)}\otimes\dots\otimes c^{(n)}$ the iterated diagonal
of $c$, $\Delta^{n-1}(c)$.
Consider the associated cosimplicial set
$\{\text{Hom}_{C-mod}(C,C^{\otimes n})\}_{n\in\mathbb{N}}$.
The coface maps $\delta_i$ and codegeneracy maps $\sigma_i$ are
given by equations~(\ref{cosimplicial associe a coendomorphism}).
Therefore for $1\leq i\leq n$,

$ev\circ \delta_0\circ ext(c_1\otimes\dots\otimes c_n)=1\otimes c_1\otimes\dots\otimes c_n
$,

$ev\circ \delta_i\circ ext(c_1\otimes\dots\otimes c_n)=c_1\otimes\dots\otimes\Delta(c_i)\otimes\dots\otimes c_n
$,

$ev\circ \delta_{n+1}\circ ext(c_1\otimes\dots\otimes c_n)= c_1\otimes\dots\otimes c_n\otimes 1
$ and

$ev\circ \sigma_{i-1}\circ ext(c_1\otimes\dots\otimes c_n)=c_1\otimes\dots\otimes\varepsilon (c_i)\otimes\dots\otimes c_n
$.

So $ext:C^{\otimes n}\buildrel{\cong}\over\rightarrow \text{Hom}_{C-mod}(C,C^{\otimes n})$
is an isomorphism of cosimplicial modules between the Cobar construction of $C$,
$\Omega C$, and the cosimplicial module associated to the operad with multiplication
$\mathcal{C}o\mathcal{E}nd_{C-mod}(C)$.
Therefore $\text{Cotor}^{*}_C(\Bbbk,\Bbbk):=H^*(\Omega C)$ is a Gerstenhaber algebra.
The composite $$C^{\otimes n}\buildrel{ext}\over\rightarrow \text{Hom}_{C-mod}(C,C^{\otimes n})
\subset \text{Hom}_{{\Bbbk}-mod}(C,C^{\otimes n})$$ admits the morphism of differential graded algebras
$$\mathcal{C}^*(C,\eta):\mathcal{C}^*_{coalg}(C,C):\rightarrow \mathcal{C}^*_{coalg}(C,{\Bbbk})=\Omega C$$
mapping $f:C\rightarrow C^{\otimes n}$ to $f(1)$ as retract.
Passing to cohomology, we obtain an injective morphism of Gerstenhaber algebras 
$$\text{Cotor}^{*}_C(\Bbbk,\Bbbk)\hookrightarrow HH^*_{coalg}(C,C)$$
which admits the morphism of graded algebras $$HH^*(C,\eta):HH^*_{coalg}(C,C)\twoheadrightarrow
\text{Cotor}^{*}_C(\Bbbk,\Bbbk)$$ as retract.
\end{proof}
\begin{propriete}\label{cotor en degre un egal primitif comme algebre de Lie}
The Lie algebra structure on $\text{Cotor}^{1}_C(\Bbbk,\Bbbk)$ given by
Theorem~\ref{cotor sous algebre de Gerstenhaber de Hochschild}
coincides with the Lie algebra of primitive elements $P(C)$ of the bialgebra $C$.
\end{propriete}
\begin{proof}
Consider the isomorphisms $ext$ and $ev$ given in the proof of Theorem~\ref{cotor sous algebre de Gerstenhaber de Hochschild}. We have:
\begin{multline}\label{Composition operad cobar}
ev\circ \circ_i \circ ( ext\otimes ext)(a_1\otimes\dots\otimes a_m\otimes b_1\otimes\dots\otimes b_n)=
\\
a_1\otimes\dots\otimes a_{i-1}\otimes a_i^{(1)}b_1\otimes\dots\otimes a_i^{(n)}b_n\otimes a_{i+1}\otimes\dots\otimes a_m.
\end{multline}
$
ev(id_C)=1_C\in C$, 
$
ev(\varepsilon)=1_{\Bbbk}\in {\Bbbk}$ and
$ ev(\Delta)=1_C\otimes 1_C\in C\otimes C$.
Therefore $ev:\text{Hom}_{C-mod}(C,C^{\otimes n})\buildrel{\cong}\over\rightarrow C^{\otimes n}$
is an isomorphism of linear operads with multiplication between 
$\mathcal{C}o\mathcal{E}nd_{C-mod}(C)$ and the operad with multiplication $\mathcal{O}$ of
~\cite[Proof of Corollary 2.9]{MenichiL:BValgaccoHa}, first considered by Gerstenhaber and Schack
~\cite[p. 65]{Gerstenhaber-Schack:algbqgad} (See also~\cite[Example 3.5]{McClure-Smith:deligneconj}).
In particular $\circ_1:\mathcal{O}(1)\otimes \mathcal{O}(1)\rightarrow \mathcal{O}(1)$
is the multiplication of $C$, $\mu:C\otimes C\rightarrow C$.
Therefore, by~(\ref{crochet en degre un}), the Lie algebra $\text{Cotor}^{1}_C(\Bbbk,\Bbbk)$ coincides with the Lie algebra
of primitive elements of $C$, denoted $P(C)$.
\end{proof}
In order to check that the Gerstenhaber algebra structure
given by Theorem~\ref{ext sous algebre de Gerstenhaber de Hochschild}
coincides with  the Gerstenhaber algebra structure on 
$\text{Ext}^{*}_A(\Bbbk,\Bbbk)$ given by Farinati and
Solotar~\cite{Farinati-Solotar:GstrcohHopf}, we give the proof of 
Theorem~\ref{ext sous algebre de Gerstenhaber de Hochschild}.
\begin{propriete}\label{propriete universelle comodule libre}
Let $C$ be a coalgebra. Let $\varepsilon:C\twoheadrightarrow {\Bbbk}$ be its counit. Let $N$ be a left $C$-comodule.
Then the linear morphism $$proj:\text{Hom}_{C-comod}(N,C)\buildrel{\cong}\over\rightarrow N^\vee,\quad F\mapsto \varepsilon\circ F,$$
is an isomorphism.
Its inverse is the linear map $lift:N^\vee\buildrel{\cong}\over\rightarrow \text{Hom}_{C-comod}(N,C)$ mapping $f:N\rightarrow {\Bbbk}$
to the composite
$N\buildrel{\Delta_N}\over\rightarrow C\otimes N\buildrel{C\otimes f}\over\rightarrow C\otimes{\Bbbk}=C$.
\end{propriete}
\begin{proof}[Proof of Theorem~\ref{ext sous algebre de Gerstenhaber de Hochschild}]
The category of left $A$-comodules, $A$-comod, is a monoidal category.
Let $M$ be a monoid in this monoidal category, i. e. $M$ is a $A$-comodule algebra~\cite[Definition III.7.1]{Kassel:quantumgrps}.
The endomorphism operad associated to $M$ is the operad
$\{\text{Hom}_{A-comod}(M^{\otimes n},M)\}_{n\in\mathbb{N}}$
with multiplication $\mu:M\otimes M\rightarrow M\in \text{Hom}_{A-comod}(M^{\otimes 2},M)$
and  $\eta:{\Bbbk}\rightarrow M\in \text{Hom}_{A-comod}(M^{\otimes 0},M)$.
The coaction of $A$ on $M^{\otimes n}$, $\Delta_{M^{\otimes n}}$, is
the composite 
$
M^{\otimes n}\buildrel{\Delta_M^{\otimes n}}\over\rightarrow (A\otimes M)^{\otimes n}\buildrel{\tau}\over\rightarrow A^{\otimes n}
\otimes M^{\otimes n}\buildrel{\mu_A\otimes M^{\otimes n} }\over\rightarrow A\otimes M^{\otimes n}$ where $\tau$ is the exchange isomorphism.
Explicitly $\Delta_{M^{\otimes n}}(a_1\otimes\dots\otimes a_n)=
a_1^{(1)} \dots a_n^{(1)}\otimes
(a_1^{(2)}\otimes\dots\otimes a_n^{(2)}$)
where $\Delta_M a_i=a_i^{(1)}\otimes a_i^{(2)}$.

We now take $M=A$. Using Property~\ref{propriete universelle comodule libre} with $C=A$ and $N=A^{\otimes n}$, we obtain that
$$proj:\text{Hom}_{A-comod}(A^{\otimes n},A)\buildrel{\cong}\over\rightarrow (A^{\otimes n})^\vee,\quad F\mapsto \varepsilon\circ F,$$
is an isomorphism.
Its inverse is the linear map $lift:(A^{\otimes n})^\vee\buildrel{\cong}\over\rightarrow \text{Hom}_{A-comod}(A^{\otimes n},A)$ mapping $f:A^{\otimes n}\rightarrow {\Bbbk}$
to $F:A^{\otimes n}\rightarrow A$ defined by
$F(a_1\otimes\dots\otimes a_n)=
a_1^{(1)} \dots a_n^{(1)}
f(a_1^{(2)}\otimes\dots\otimes a_n^{(2)}$).

Therefore the composite $$(A^{\otimes n})^\vee\buildrel{lift}\over\rightarrow \text{Hom}_{A-comod}(A^{\otimes n},A)
\subset \text{Hom}_{{\Bbbk}-mod}(A^{\otimes n},A)$$
coincides with the section of 
$\mathcal{C}^*(A,\varepsilon):\mathcal{C}^*(A,A)\rightarrow BA^\vee
$
defined by Farinati and
Solotar~\cite[p. 2862]{Farinati-Solotar:GstrcohHopf}.

Consider the associated cosimplicial set
$\{\text{Hom}_{A-comod}(A^{\otimes n},A)\}_{n\in\mathbb{N}}$.
The coface maps $\delta_i$ and codegeneracy maps $\sigma_i$ are
given by equations~(\ref{cosimplicial associe a endomorphism}).
Therefore for $1\leq i\leq n$,

$proj\circ \delta_0\circ lift(f)(a_1\otimes\dots\otimes a_{n+1})=
\varepsilon(a_1)f(a_2\otimes\dots\otimes a_{n+1})
$,

$proj\circ \delta_i\circ lift(f)(a_1\otimes\dots\otimes a_{n+1})=f(a_1\otimes\dots\otimes a_ia_{i+1}\otimes\dots\otimes a_{n+1})
$,

$proj\circ \delta_{n+1}\circ lift(f)(a_1\otimes\dots\otimes a_{n+1})= f(a_1\otimes\dots\otimes a_n)\varepsilon(a_{n+1})
$ and

$proj\circ \sigma_{i-1}\circ lift(f)(a_1\otimes\dots\otimes a_{n-1})=
f(a_1\otimes\dots\otimes a_{i-1}\otimes 1_A\otimes a_i\otimes\dots\otimes a_n)
$.

So $lift:(A^{\otimes n})^\vee\buildrel{\cong}\over\rightarrow
\text{Hom}_{A-comod}(A^{\otimes n},A)$
is an isomorphism of cosimplicial modules between the dual of the bar construction of $A$,
$BA^\vee$, and the cosimplicial module associated to the operad with multiplication
$\mathcal{E}nd_{A-comod}(A)$.
Therefore $\text{Ext}^{*}_A(\Bbbk,\Bbbk):=H^*(BA^\vee)$ is a Gerstenhaber algebra.
\end{proof}
Let $A$ be an algebra and $M$ be a $A$-bimodule.
The cocycles of degre $1$ of the Hochschild complex $\mathcal{C}^*(A,M)$
are exactly the module of derivations $\text{Der}(A,M)$.
A linear map $f:A\rightarrow M$ is a {\em derivation} if and only if 
$\forall a$, $b\in A$, $f(ab)=f(a)b+af(b)$.
The boundaries of degre $1$ of  $\mathcal{C}^*(A,M)$ are the inner derivations, i. e.
the linear maps $f:A\rightarrow M$, $a\mapsto am-ma$, where $m$ is a given element of
$M$.
The degre $1$ component of Hochschild cohomology, $HH^1(A,M)$,
can be identified with the quotient $\text{Der}(A,M)/\{\text{inner derivations}\}$
~\cite[1.5.2]{LodayJ.:cych}.
In particular, suppose that $A$ has an augmentation $\varepsilon:A\rightarrow {\Bbbk}$.
Then $\text{Ext}_A^1({\Bbbk},{\Bbbk})=HH^1(A,{\Bbbk})=\text{Der}(A,{\Bbbk})$.
\begin{propriete}\label{ext en degre un egal dual indecomposable comme algebre de Lie}
Let $A$ be a bialgebra. The inclusion of Lie algebra 
$\text{Ext}_A^1({\Bbbk},{\Bbbk})\hookrightarrow HH^1(A,A)$ given by
Theorem~\ref{ext sous algebre de Gerstenhaber de Hochschild}
can be identified with the following composite of Lie algebra morphisms
$$
\text{Der}(A,{\Bbbk})\buildrel{i}\over\hookrightarrow 
\text{Der}(A,A)\buildrel{q}\over\twoheadrightarrow \text{Der}(A,A)/\{inner derivations\}.
$$
Here $q$ is the obvious quotient map and $i$ is the inclusion of the Lie algebra of ``differentiations'' of $A$
into the Lie algebra of derivations of $A$ given
by~\cite[p. 36]{Hochschild:alggroupandliealgebra}.
\end{propriete}
Let $G$ be an affine algebraic group.
Then the {\em algebra of polynomial functions} on $G$, ${\mathcal{P}(G)}$,
is a commutative Hopf algebra.
By definition~\cite[p. 36]{Hochschild:alggroupandliealgebra}, the Lie algebra
of $G$ is $\text{Ext}_{\mathcal{P}(G)}^1({\Bbbk},{\Bbbk})=\text{Der}(\mathcal{P}(G),{\Bbbk})$.

Let $G$ be a Lie group. The algebra of smooth maps on $G$, $C^\infty(G)$, 
is a module algebra over the group ring $\mathbb{R}[G]$, but is not a bialgebra
(except when $G$ is finite and discrete).
However there is still an analogue of the inclusion $i$: the composite
$$
T_e(G)\build\rightarrow_{lift}^{\cong}\text{Hom}_{mod-\mathbb{R}[G]}(C^\infty(G),C^\infty(G))\cap\text{Der}(C^\infty(G))\subset \text{Der}(C^\infty(G)).
$$
Here $lift$ is the isomorphism between the tangent space and the right
invariant vector fields on $G$.
\begin{proof}
Consider the inverse isomorphisms
$proj:\text{Hom}_{A-comod}(A^n,A)\buildrel{\cong}\over\rightarrow (A^{\otimes n})^\vee$
and $lift: (A^{\otimes n})^\vee\buildrel{\cong}\over\rightarrow \text{Hom}_{A-comod}(A^n,A)$
given in the proof of Theorem~\ref{ext sous algebre de Gerstenhaber de Hochschild}.
Let $\mathcal{O}$ denote the linear operad with multiplication
such that $proj:
\mathcal{E}nd_{A-comod}(A)\buildrel{\cong}\over\rightarrow \mathcal{O}
$
is an isomorphism of linear operads with multiplication.
Explicitly, for $f\in\mathcal{O}(m)=(A^{\otimes m})^\vee$
and $g\in\mathcal{O}(n)=(A^{\otimes n})^\vee$, $f\circ_i g$ is given by
\begin{multline*}
f\circ_i g(a_1\otimes\dots\otimes a_{m+n-1})=\\
f(a_1\otimes\dots\otimes a_{i-1}\otimes a_i^{(1)}\dots a_{i+n-1}^{(1)}g(a_i^{(2)}\otimes\dots a_{i+n-1}^{(2)})\otimes a_{i+n}\otimes\dots\otimes a_{m+n-1}
\end{multline*}
where $\Delta a_j=a_j^{(1]}\otimes a_j^{(2)}$.
The identity element of $\mathcal{O}$ is the counit of $A$,
$\varepsilon\in A^\vee=\mathcal{O}(1)$.
The multiplication of $\mathcal{O}$ is the composite $\varepsilon\circ\mu\in
(A\otimes A)^\vee=\mathcal{O}(2)$ and the unit is $id_{\Bbbk}\in (A^{\otimes 0})^\vee=\mathcal{O}(0)$.
In particular, $\circ_1:\mathcal{O}(1)\otimes\mathcal{O}(1)\rightarrow \mathcal{O}(1)$
is the multiplication of $A^\vee$,
$\mu_{A^\vee}:A^\vee\otimes A^\vee\rightarrow A^\vee$
obtained by dualizing the diagonal of $A$.
Therefore, by~(\ref{crochet en degre un}), the Lie algebra $\mathcal{C}^1(\mathcal{O})$ is just the Lie algebra
associated to the associative algebra $A^\vee$.
The composite $\mathcal{O}\build\rightarrow_{\cong}^{lift}
\mathcal{E}nd_{A-comod}(A)\subset \mathcal{E}nd_{{\Bbbk}-mod}(A)
$
is an injective morphism of linear operads with multiplication.
Therefore
 this composite 
$\mathcal{C}^*(\mathcal{O})\build\rightarrow_{\cong}^{lift}
\mathcal{C}^*(\mathcal{E}nd_{A-comod}(A))\subset \mathcal{C}^*(\mathcal{E}nd_{{\Bbbk}-mod}(A))
$
is an injective morphism of differential graded Lie algebras.
In degre $1$, this composite
$\mathcal{O}(1)\hookrightarrow  \mathcal{E}nd_{{\Bbbk}-mod}(A)(1)=\text{Hom}_{{\Bbbk}-mod}(A,A)$
is the injective morphism of (associative) algebras, mapping $f:A\rightarrow {\Bbbk}$
to $(A\otimes f)\circ \Delta_A$,
given by~\cite[I.Proposition 2.1]{Hochschild:alggroupandliealgebra}.
Restricted at the cycles in degre $1$, this composite gives the injective
morphism of Lie algebras
$
\text{Der}(A,{\Bbbk})\buildrel{i}\over\hookrightarrow 
\text{Der}(A,A)
$
considered in~\cite[p. 36]{Hochschild:alggroupandliealgebra}.
\end{proof}
Let us prove that Theorem~\ref{ext sous algebre de Gerstenhaber de Hochschild}
is the dual of Theorem~\ref{cotor sous algebre de Gerstenhaber de Hochschild}.
\begin{lemma}\label{comparaison bar cobar hochschild}
Let $C$ be a coalgebra with coaugmentation $\eta:{\Bbbk}\rightarrow C$.
Let $A=C^\vee$ be the dual algebra with augmentation $\varepsilon:A\rightarrow {\Bbbk}$.
Then

i) the linear map
$\Gamma: \mathcal{C}o\mathcal{E}nd_{{\Bbbk}-mod}(C))\rightarrow \mathcal{E}nd_{{\Bbbk}-mod}(A)$,
mapping $f:C\rightarrow C^{\otimes n}$ to the composite
$
A^{\otimes n}\rightarrow (C^{\otimes n})^\vee\buildrel{f^\vee}\over\rightarrow A
$, is a morphism of linear operads with multiplication,

ii) the linear map $\phi:\Omega C\rightarrow (BA)^\vee$, such that
$\phi(c_1\otimes\dots\otimes c_n)$ is the form on $
A^{\otimes n}$, mapping $\varphi_1\otimes\dots\otimes\varphi_n$ to the product
$\varphi_1(c_1)\dots\varphi_n(c_n)$,
is a morphism of differential graded algebras.

iii) We have the commutative diagram of differential graded algebras
$$\xymatrix{
\mathcal{C}_{coalg}^*(C,C)\ar[r]^{\mathcal{C}_{coalg}^*(C,\eta)}\ar[d]_\Gamma
& \Omega C\ar[d]^\phi\\
\mathcal{C}^*(A,A)\ar[r]_{\mathcal{C}^*(A,\varepsilon)}
& (BA)^\vee
}
$$
If $C$ is ${\Bbbk}$-free of finite type then both $\Gamma$ and $\phi$ are isomorphisms.
\end{lemma}
\begin{proof}
Let $\psi:V^\vee\otimes W^\vee\rightarrow (V\otimes W)^\vee$ be the linear map,
mapping the tensor product $\varphi_1\otimes \varphi_2$ of a form on $V$ and
of a form on $W$, to the form on $V\otimes W$, also denoted $\varphi_1\otimes \varphi_2$,
mapping $v\otimes w$ to the product $\varphi_1(v)\varphi_2(w)$.
The functor $^\vee$ from the opposite category of ${\Bbbk}$-modules
to the category of ${\Bbbk}$-modules, mapping a ${\Bbbk}$-module $V$, to its dual
$V^\vee:=\text{Hom}(V,{\Bbbk})$ is a monoidal functor.
Therefore by applying Property~\ref{functor monoidal induit morphism d'operades},
we obtain i). ii) is well-known and iii) is easy to check.
\end{proof}
Note that in~\cite{Felix-Menichi-Thomas:GerstduaiHochcoh}, together with Felix
and Thomas, we gave a different proof that $H^*(\Gamma):HH^*_{coalg}(C,C)\rightarrow HH^*(A,A)$ is a morphism
of Gerstenhaber algebras.
\begin{proposition}\label{dualite entre inclusions Gerstenhaber}
Let $C$ be a bialgebra ${\Bbbk}$-free of finite type.
Let $A$ be the dual bialgebra.
Then the inclusions of Gerstenhaber algebras given by
Theorems~\ref{cotor sous algebre de Gerstenhaber de Hochschild} and~\ref{ext sous algebre de Gerstenhaber de Hochschild} fit into the commutative diagram of Gerstenhaber algebras.
$$
\xymatrix{
\text{Cotor}_C^*({\Bbbk},{\Bbbk})\ar[r]\ar[d]_{H^*(\phi)}^\cong
&HH_{coalg}^*(C,C)\ar[d]^{H^*(\Gamma)}_\cong\\
\text{Ext}_A^*({\Bbbk},{\Bbbk})\ar[r]
&HH^*(A,A)
}
$$
\end{proposition}
\begin{proof}
Since $C$ is an algebra ${\Bbbk}$-free of finite type,
the dualizing functor $^\vee$, defined in the proof of
Lemma~\ref{comparaison bar cobar hochschild},
restrict to a functor $F$ from the opposite category of left $C$-modules
to the category of left $A$-comodules.
If $M$ and $N$ are left $C$-modules, 
 $\psi:M^\vee\otimes N^\vee\rightarrow (M\otimes N)^\vee$
is a morphism of left $A$-comodules. Therefore by Property~\ref{functor monoidal induit morphism d'operades},
we obtain the morphism of linear operads with multiplication $\Gamma_F:
\mathcal{C}o\mathcal{E}nd_{C-mod}(C))\rightarrow \mathcal{E}nd_{A-comod}(A)$.
Consider the two commutatives squares
$$
\xymatrix{
\mathcal{C}o\mathcal{E}nd_{C-mod}(C)\ar[r]^{i_C}\ar[d]_{\Gamma_F}^\cong
& \mathcal{C}o\mathcal{E}nd_{{\Bbbk}-mod}(C)\ar[d]_{\Gamma}^\cong
\ar[r]^{\mathcal{C}_{coalg}^*(C,\eta)}
&\Omega C\ar[d]^{\phi}_\cong\\
\mathcal{E}nd_{A-comod}(A)\ar[r]^{i_A}
&\mathcal{E}nd_{{\Bbbk}-mod}(A)\ar[r]_{\mathcal{C}^*(A,\varepsilon)}
&(BA)^\vee }$$
The left square commutes by definition of $\Gamma_F$ since the two
horizontal maps $i_C$ and $i_A$ are just the inclusions.
Part iii) of Lemma~\ref{comparaison bar cobar hochschild} says that the right
square commutes.
The composite $\mathcal{C}o\mathcal{E}nd_{C-mod}(C)\buildrel{i_C}\over\hookrightarrow
\mathcal{C}o\mathcal{E}nd_{{\Bbbk}-mod}(C)\buildrel{\mathcal{C}_{coalg}^*(C,\eta)}\over\rightarrow\Omega C$ is the isomorphism $ev$ considered in the proof of Theorem~\ref{cotor sous algebre de Gerstenhaber de Hochschild}.
The composite $\mathcal{E}nd_{A-comod}(A)\buildrel{i_A}\over\hookrightarrow
\mathcal{E}nd_{{\Bbbk}-mod}(A)\buildrel{\mathcal{C}^*(A,\varepsilon)}\over\rightarrow(BA)^\vee$ is the isomorphism $proj$ considered in the proof of Theorem~\ref{ext sous algebre de Gerstenhaber de Hochschild}.
Therefore,we have the commutative square of linear operads with multiplication
$$
\xymatrix{
\Omega C\ar[r]^{i_c\circ ev^{-1}}\ar[d]_{\phi}^\cong
&\mathcal{C}_{coalg}^*(C,C)\ar[d]^{\Gamma}_\cong\\
(BA)^{\vee}\ar[r]^{i_A\circ proj^{-1}}
&\mathcal{C}^*(A,A)
}
$$
Applying homology, we obtain the Proposition.
\end{proof}
Let $H$ be a finite dimensional Hopf algebra. %with invertible antipode.
Let $D(H)$ be the Drinfeld double of $H$.
Then Taillefer~\cite{Taillefer:InjectiveHopf} proved that
the Gerstenhaber-Schack cohomology of $H$, $H_{GS}(H,H)$
is isomorphic as graded algebras to $\text{Ext}_{D(H)^{op}}(\mathbb{F},\mathbb{F})$.
Since $D(H)$ is a Hopf algebra, by Theorem~\ref{ext sous algebre de Gerstenhaber de Hochschild}, Farinati and
Solotar~\cite{Farinati-Solotar:GstrcohHopf} have obtained a Gerstenhaber
algebra structure on $\text{Ext}_{D(H)^{op}}(\mathbb{F},\mathbb{F})=H_{GS}(H,H)$.
But Taillefer using a
braiding~\cite[Beginning of Section 5]{Taillefer:InjectiveHopf}
shows that the Lie bracket in this Gerstenhaber algebra structure
is trivial. The Drinfeld double $D(H)$ is a braided Hopf algebra.
Therefore, following the proof of Taillefer, it should be easy to
prove 
\begin{conjecture}\label{conjecture crochet trivial pour bialgebre tresse}
Let $A$ be braided bialgebra.
Then the Lie algebra of the Gerstenhaber algebra $\text{Ext}_A^*(\mathbb{F},\mathbb{F})$ given by Theorem~\ref{ext sous algebre de Gerstenhaber de Hochschild}
is trivial.
\end{conjecture}
\begin{proof}[Proof when $A$ is a cocommutative Hopf algebra]
Let $A$ be a cocommutative Hopf algebra. Since $A$ is cocommutative,
the antipode $S$ is involutive. Therefore by Theorem~\ref{ext BV algebre},
$\text{Ext}^*_A(\mathbb{F},\mathbb{F})$ is a Batalin-Vilkovisky algebra.
By~\cite[Theorem 4.1]{Khalkhali-Rangipour:newcyclic}, the operator $B$ of this Batalin-Vilkovisky algebra is trivial. Therefore by~(\ref{relation BV}),
the Lie bracket is null.
\end{proof}
In~\cite{shoikhet:hopftetra}, Shoikhet mentions the following conjecture of Kontsevich.
\begin{conjecture}(Kontsevich)
Let $H$ be a bialgebra.
Then $H_{GS}(H,H)$ is a $3$-algebra, i. e. \cite[Theorem p. 26-7]{Markl-Shnider-Stasheff:opeatp} an algebra over the homology of the little $3$-cubes operad, $\mathcal{C}_3$.
\end{conjecture}
Shoikhet~\cite[Corollary 0.3]{shoikhet:hopftetra}
has announced that the proof of this conjecture when $H$ is a Hopf algebra.
We formulate the following related conjecture:
\begin{conjecture}\label{conjecture ext 3-algebre}
Let $A$ be braided bialgebra.
Then $\text{Ext}_A^*(\mathbb{F},\mathbb{F})$ is a $3$-algebra.
\end{conjecture}
If $A$ is cocommutative, again this Lie algebra bracket (of degre $-2$ this time)
should vanish: modulo $p$, only the Steenrod or Dyer-Lashoff operations
on $\text{Ext}_A^*(\mathbb{F},\mathbb{F})$ (Remark~\ref{operation sur Steenrod sur cohomologie de Hochschild} and~\cite[Theorem 11.8]{maygeneral}) should be non trivial.

As an algebraic topologist, we find the following Corollaries of
Theorem~\ref{cotor sous algebre de Gerstenhaber de Hochschild}
and Theorem~\ref{ext sous algebre de Gerstenhaber de Hochschild}, highly interesting.
\begin{corollary}\label{lacet double sous algebre Gerstenhaber de Hochschild}
Let $X$ be a $2$-connected pointed topological space.
Denote by $\Omega_M X$ the pointed Moore loops on $X$.
Then the homology of the double loop spaces on $X$, $H_*(\Omega^2X)$, equipped
with the Pontryagin product,
is a sub Gerstenhaber algebra of $HH^*_{coalg}(S_*(\Omega_M X),S_*(\Omega_M X))$,
the Hochschild cohomology of the coalgebra $S_*(\Omega_M X)$.
\end{corollary}
\begin{proof}
The bialgebra $C$ in Theorem~\ref{cotor sous algebre de Gerstenhaber de Hochschild} can be differential graded.
Since $\Omega_M X$ is a topological monoid, the (reduced normalized)
singular chains on $\Omega_M X$ form a differential graded bialgebra
$C=S_*(\Omega_M X)$.
Therefore, by Theorem~\ref{cotor sous algebre de Gerstenhaber de Hochschild},
$\text{Cotor}_{S_*(\Omega_M X)}({\Bbbk},{\Bbbk})$ is a sub Gerstenhaber
algebra of $HH^*_{coalg}(S_*(\Omega_M X),S_*(\Omega_M X))$.
By Adams cobar equivalence, there is an isomorphism of graded algebras
$\text{Cotor}_{S_*(\Omega_M X)}({\Bbbk},{\Bbbk})\cong H_*(\Omega_M\Omega_M X)$.
The inclusion of the (ordinary) pointed loops into the Moore loops
$\Omega X\buildrel{\approx}\over\hookrightarrow \Omega_M X$ is a both a homotopy
equivalence~\cite[p. 112-3]{Whitehead:eltsoht} and a morphism of $H$-spaces.
So as graded algebras, $H_*(\Omega_M\Omega_M X)$ is isomorphic to
$H_*(\Omega^2X)$.
\end{proof}
Corollary~\ref{lacet double sous algebre Gerstenhaber de Hochschild}
gives in particular a Gerstenhaber algebra structure
on $H_*(\Omega^2X)$ extending the Pontryagin product.
Of course, we believe that this Gerstenhaber algebra
structure coincides with the usual one given by Cohen
in~\cite{Cohen-Lada-May:homiterloopspaces}:
\begin{conjecture}\label{conjecture iso homologie lacets double cotor}
Let $X$ be a $2$-connected pointed topological space.
There is an isomorphism of Gerstenhaber algebras between the Gerstenhaber
algebra $\text{Cotor}_{S_*(\Omega_M X)}({\Bbbk},{\Bbbk})$ given by
Theorem~\ref{cotor sous algebre de Gerstenhaber de Hochschild}
 and the Gerstenhaber
algebra $H_*(\Omega^2X)$ given by Cohen
in~\cite{Cohen-Lada-May:homiterloopspaces}.
\end{conjecture}
Recall that the Gerstenhaber algebra on $H_*(\Omega^2X)$ is usually defined as follows:
the little $2$-cube operad $\mathcal{C}_2$ acts on the double loop space,
$\Omega^2 X$. So its homology $H_*(\Omega^2X)$ is an algebra over the homology of $\mathcal{C}_2$, i. e. is a Gerstenhaber algebra by
Cohen~\cite{Cohen-Lada-May:homiterloopspaces}.
\begin{corollary}\label{cohomology d'un espace sous algebre Gerstenhaber de Hochschild}
Let $X$ be a path-connected topological space.
Denote by $\Omega_M X$ the pointed Moore loops on $X$.
Then the cohomology of $X$, $H^*(X)$, equipped
with the cup product,
is a sub Gerstenhaber algebra of $HH^*(S_*(\Omega_M X),S_*(\Omega_M X))$,
the Hochschild cohomology of the algebra $S_*(\Omega_M X)$.
\end{corollary}
\begin{proof}
By Theorem~\ref{ext sous algebre de Gerstenhaber de Hochschild}
applied to $A=S_*(\Omega_M X)$,
$\text{Ext}^*_{S_*(\Omega_M X)}({\Bbbk},{\Bbbk})$ is a sub Gerstenhaber algebra of $HH^*(S_*(\Omega_M X),S_*(\Omega_M X))$.
Applying homology to~\cite[Theorem 7.2 ii)]{Felix-Halperin-Thomas:dgait},
gives the natural isomorphism of graded algebras
$$
H^*(X)\cong \text{Ext}^*_{S_*(\Omega_M X)}({\Bbbk},{\Bbbk}).
$$
\end{proof}
We believe that the Lie bracket on $
H^*(X)\cong \text{Ext}^*_{S_*(\Omega_M X)}({\Bbbk},{\Bbbk})
$ given by Corollary~\ref{cohomology d'un espace sous algebre Gerstenhaber de Hochschild}
is trivial since $S_*(\Omega_M X)$ is cocommutative up to homotopy in some
$E_\infty$-sense.

\section{Batalin-Vilkovisky algebras}\label{algebres de Batalin-Vilkovisky}
\begin{numerotation}\label{definition BV algebre}
A {\it Batalin-Vilkovisky algebra} is a Gerstenhaber algebra
$A$ equipped with a degree $-1$ linear map $B:A^{i}\rightarrow A^{i-1}$
such that $B\circ B=0$ and
\begin{equation}\label{relation BV} 
\{a,b\}=(-1)^{\vert a\vert}\left(B(a\cup b)-(B a)\cup b-(-1)^{\vert
  a\vert}a\cup(B b)\right)
\end{equation}
for $a$ and $b\in A$.
\end{numerotation}
\begin{definition}
A {\it cyclic operad} is a non-$\Sigma$-operad $\mathcal{O}$
equipped with linear maps $\tau_n:O(n)\rightarrow O(n)$
for $n\in\mathbb{N}$ such that
\begin{equation}
\forall n\in\mathbb{N},\quad\tau_n^{n+1}=id_{O(n)},
\end{equation}
\begin{equation}\label{deux cyclic}
 \forall m\geq 1,n\geq 1,
\quad\tau_{m+n-1}(f\circ_1 g)= \tau_n g\circ_n \tau_m f,
\end{equation}
\begin{equation}\label{trois cyclic}
\forall m\geq 2, n\geq 0, 2\leq i\leq m,
\quad\tau_{m+n-1}(f\circ_i g)=\tau_m f\circ_{i-1}g,
\end{equation}
for each $f\in O(m)$ and $g\in O(n)$. In particular, we have
$\tau_1 id=id$.
\end{definition}
\begin{definition}~\cite{MenichiL:BValgaccoHa}\label{operade cyclique avec multiplication}
A {\it cyclic operad with multiplication} is an operad
which is both an operad with multiplication and a cyclic operad
such that 
\begin{equation*}
\tau_2\mu=\mu.
\end{equation*}
\end{definition}
\begin{theorem}~\cite{MenichiL:BValgaccoHa}\label{Theoreme BV algebre}
If $\mathcal{O}$ is a cyclic operad with a multiplication then

a) the structure of cosimplicial module on $\mathcal{O}$ extends to a structure
of cocyclic module and

b) the Connes coboundary map $B$ on $\mathcal{C}^{*}(\mathcal{O})$
induces a natural structure of Batalin-Vilkovisky algebra
on the Gerstenhaber algebra $H^*(\mathcal{C}^{*}(\mathcal{O}))$.
\end{theorem}
\begin{theorem}\label{crochet sur cohomologie cyclique lambda}
Let $\mathcal{O}$ be a linear cyclic operad with multiplication.
Consider the associated cocyclic module.
Then the cyclic cochains $\mathcal{C}_\lambda^*(\mathcal{O})$ forms 
a subcomplex of $\mathcal{C}^*(\mathcal{O})$, stable under the Lie bracket od degre $-1$.
In particular, the cyclic cohomology $HC^*_\lambda(\mathcal{C}^*(\mathcal{O}))$
has naturally a graded Lie algebra structure of degre $-1$.
\end{theorem}
\begin{proof}
Let $\mathcal{O}$ be a linear cyclic operad.
Let $f\in \mathcal{C}^m_\lambda(\mathcal{O})$ and $g\in \mathcal{C}^n_\lambda(\mathcal{O})$.
Using~(\ref{deux cyclic}), (\ref{trois cyclic}) and the change of variable $i'=i-1$ for the
first equation and using $\tau_m(f)=(-1)^m f$ and $\tau_m(g)=(-1)^n g$ for the second equation,
we have
\begin{multline*}
\tau_{m+n-1}(f\overline{\circ}g )=(-1)^{(m-1)(n-1)}
\left(\tau_n g\circ_n\tau_m f+\sum_{i'=1}^{m-1}(-1)^{(n-1)i'}\tau_mf\circ_{i'} g\right)\\
=(-1)^{(m-1)(n-1)}\left((-1)^{m+n}g\circ_n f
+(-1)^{m+n-1}\sum_{i=1}^{m-1}(-1)^{(n-1)(i-1)}f\circ_{i} g)
\right).
\end{multline*}
By symmetry
$$
(-1)^{(m-1)(n-1)}\tau_{m+n-1}(g\overline{\circ}f )
=(-1)^{m+n}f\circ_m g
+(-1)^{m+n-1}\sum_{i=1}^{n-1}(-1)^{(m-1)(i-1)}g\circ_{i} f.
$$
Therefore $\tau_{m+n-1}\{f,g\}=(-1)^{m+n-1}\{f,g\}$.

Suppose now that $\mathcal{O}$ has an associative multiplication $\mu$
such that $\tau_2\mu=\mu$.
Since $\mu\in \mathcal{C}^2_\lambda(\mathcal{O})$, we have just proved above that
for any $g\in \mathcal{C}^n_\lambda(\mathcal{O})$, the differential of $g$, $d(g)=\{\mu,g\}\in\mathcal{C}_\lambda (\mathcal{O})$. 
\end{proof}
\begin{remark}
Let $\mathcal {O}$ be a cyclic operad. Then $\tau_1:\mathcal {O}(1)\rightarrow
\mathcal {O}(1)$ is an involutive morphism of anti-algebras.
And $\mathcal{C}^1_\lambda(\mathcal{O})=\text{Ker} (\tau_1+Id:\mathcal {O}(1)\rightarrow
\mathcal {O}(1))$ is a sub Lie algebra of the Lie algebra associated
to the associative algebra $\mathcal {O}(1)$(Compare with Remark~\ref{crochet en degre un}).
\end{remark}
\begin{remark}\label{crochet sur cohomologie cyclique negative}
In~\cite[Corollary 1.5]{MenichiL:BValgaccoHa},
motivated by applications to string topology
~\cite{Chas-Sullivan:stringtop},
we proved that the negative cyclic cohomology of a cyclic operad $\mathcal{O}$
with multiplication, $HC^*_-(\mathcal{C}^{*}(\mathcal{O}))$, has a Lie bracket
of degre $-2$.
\end{remark}
\begin{remark}\label{crochet sur cohomologie cyclique ordinaire}
The (ordinary) cyclic cohomology of $\mathcal{O}$, $HC^*(\mathcal{C}^{*}(\mathcal{O}))$,
has also a Lie bracket of degre $-1$.
This was stated only in the case of the cyclic cohomology of the group ring
${\Bbbk}[G]$ of a finite group $G$~\cite[Theorem 67 a)]{Chataur-Menichi:stringclass}.
But the proof of~\cite[Theorem 67 a)]{Chataur-Menichi:stringclass}
works for any cyclic operad with multiplication.
\end{remark}
\begin{remark}
The proof of Theorem~\ref{crochet sur cohomologie cyclique lambda}
is a lot more simple than the proofs of remarks~\ref{crochet sur cohomologie cyclique negative} and~\ref{crochet sur cohomologie cyclique ordinaire}.
Indeed, the proofs of remarks~\ref{crochet sur cohomologie cyclique negative} and~\ref{crochet sur cohomologie cyclique ordinaire} use that
$H^*(\mathcal{C}^{*}(\mathcal{O}))$ is a Batalin-Vilkovisky algebra (Theorem~\ref{Theoreme BV algebre}).
On the contrary, in the proof of Theorem~\ref{crochet sur cohomologie cyclique lambda},
we don't even use that $H^*(\mathcal{C}^{*}(\mathcal{O}))$ is a Gerstenhaber algebra:
we use only the Lie algebra on $\mathcal{C}^{*}(\mathcal{O})$.
\end{remark}
If our ground ring ${\Bbbk}$ contains $\mathbb{Q}$,
there is a natural isomorphim~\cite[p. 72]{LodayJ.:cych}
$$
HC_\lambda^n(\mathcal{C}^{*}(\mathcal{O}))\buildrel{\cong}\over\rightarrow
HC^n(\mathcal{C}^{*}(\mathcal{O})).
$$
This isomorphism obviously should be compatible with the brackets.

Recall the following well-known result in string topology.
\begin{corollary}(\cite{Tradler:bvalgcohiiip},\cite[Theorem 1.6]{MenichiL:BValgaccoHa})\label{Cohomologie Hochschild algebre symetrique BV}
Let $A$ be a symmetric Frobenius algebra (Definition~\ref{def algebre de Frobenius}).
Then its Hochschild cohomology $HH^*(A,A)$ is a Batalin-Vilkovisky algebra.
\end{corollary}
We need to sketch our proof given in~\cite{MenichiL:BValgaccoHa}.
\begin{proof}%[Proof of Corollaries~\ref{Cohomologie Hochschild algebre symetrique BV} and~\ref{crochet sur la Cohomologie cyclic rationelle}]
Let $\Theta:A\buildrel{\cong}\over\rightarrow A^\vee$ be an isomorphism
of $A$-bimodules given by the symmetric Frobenius algebra structure on $A$.
Then $\mathcal{C}^*(A,\Theta):\mathcal{C}^*(A,A)\buildrel{\cong}\over\rightarrow\mathcal{C}^*(A,A^\vee)$ is an isomorphism of cosimplicial
modules.
Let $Ad:\mathcal{C}^*(A,A^\vee)\buildrel{\cong}\over\rightarrow
\mathcal{C}_*(A,A)^\vee$ be the adjunction map~\cite[(4.1)]{MenichiL:BValgaccoHa}
which associates to any $g\in\text{Hom}(A^n,A^\vee)$, the linear map
$Ad(g):A\otimes A^{\otimes n}\rightarrow {\Bbbk}$ given by
$Ad(g)(a_0\otimes\dots\otimes a_n)=g(a_1,\dots,a_n)(a_0)$.
Then $Ad:\mathcal{C}^*(A,A^\vee)\buildrel{\cong}\over\rightarrow
\mathcal{C}_*(A,A)^\vee$ is an isomorphism of cosimplicial modules.
By~\cite[Proof of Theorem 1.6]{MenichiL:BValgaccoHa}
% ~\cite[(1.8.4) and last sentence of (1.7)]{Getzler-Kapranov:modularoperad}
$$\mathcal{C}^*(A,A)\build\rightarrow_\cong^{\mathcal{C}^*(A,\Theta)}
\mathcal{C}^*(A,A^\vee)\build\rightarrow_\cong^{Ad}
\mathcal{C}_*(A,A)^\vee$$
equipped with the $\tau_n$~\cite[(4.2)]{MenichiL:BValgaccoHa}
is a cyclic operad with multiplication.
Using Theorem~\ref{Theoreme BV algebre},
$HH^*(A,A)\buildrel{\cong}\over\rightarrow HH^*(A,A^\vee)$ is a Batalin-Vilkovisky algebra.
\end{proof}
If instead of using Theorem~\ref{Theoreme BV algebre}, we apply 
Theorem~\ref{crochet sur cohomologie cyclique lambda} in the previous proof,
we obtain the following Corollary:
\begin{corollary}\label{crochet sur la Cohomologie cyclic rationelle}
Let $A$ be a symmetric Frobenius algebra.
Then its cyclic cohomology $HC^*_\lambda(A)$ (in the sense of~\cite[2.4.2]{LodayJ.:cych}), is a
graded Lie algebra of degree $-1$.
\end{corollary}
We wonder if this Corollary is not a
particular simple case of~\cite[Proposition 2.11]{Hamilton-Lazarev:symplecticstring}?

In~\cite{MenichiL:BValgaccoHa}, our main objective was the following result:
\begin{corollary}\cite[Theorem 1.1]{MenichiL:BValgaccoHa}\label{Cobar algebre de Hopf BV}
Let $H$ be a Hopf algebra endowed with a modular pair in
involution $(\Character,1)$  (Definition~\ref{definition modular pair in involution}). Then the Connes-Moscovici cocyclic on the Cobar construction
on $H$,
defines a Batalin-Vilkovisky
algebra structure on $\text{Cotor}^{*}_{H}(\Bbbk,\Bbbk)$.
\end{corollary}
\begin{proof}
A computation~\cite[Section 5]{MenichiL:BValgaccoHa} shows
that the operad with multiplication $\mathcal{C}o\mathcal{E}nd_{H-mod}(H)\cong \Omega H$
considered in the proof of Theorem~\ref{cotor sous algebre de Gerstenhaber de Hochschild}
equipped with the $\tau_n$ defined by Connes and Moscovici, is cyclic.
Therefore, by Theorem~\ref{Theoreme BV algebre}, its homology
$\text{Cotor}^{*}_{H}(\Bbbk,\Bbbk)$ is a Batalin-Vilkovisky algebra.
\end{proof}
If instead of using Theorem~\ref{Theoreme BV algebre}, we apply 
Theorem~\ref{crochet sur cohomologie cyclique lambda} in the previous proof,
we obtain the following Corollary:
\begin{corollary}\label{crochet sur Hopf cyclic cohomologie}
Let $H$ be a Hopf algebra endowed with a modular pair in
involution $(\Character,1)$ (Definition~\ref{definition modular pair in involution}).
Then its cyclic cohomology,
$HC^*_{(\Character,1)}(H)$, is a graded Lie algebra of degree $-1$.
\end{corollary}
\section{Batalin-Vilkovisky algebra structure on $\text{Ext}^{*}_{H}(\Bbbk,\Bbbk)$}\label{section ext BV-algebre}
Everybody is more familiar with an algebra $A$ than with a coalgebra $C$.
And therefore, one usually prefers the Exterior product $\text{Ext}^{*}_{A}(\Bbbk,\Bbbk)$
instead of the Cotorsion product  $\text{Cotor}^{*}_{C}(\Bbbk,\Bbbk)$.
The goal of this section is to give the duals of Corollaries~\ref{Cobar algebre de Hopf BV}
and~\ref{crochet sur Hopf cyclic cohomologie}, Theorem~\ref{ext BV algebre} below.
Taillefer~\cite{Taillefer:cyclicHopf}, Khalkhali and Rangipour~\cite{Khalkhali-Rangipour:newcyclic}
developped a theory dual to Connes-Moscovici cyclic cohomology of Hopf algebras.
First, we are going to explain this duality.
\begin{proposition}\label{dualite Connes-Moscovici Khalkhali-Rangipour}
Let $K$ be a finite dimensional Hopf algebra with a modular pair in involution
$(\Character, \sigma )$ in the sense
of Khalkhali-Rangipour~\cite[(1)]{Khalkhali-Rangipour:newcyclic}.
Then i) its dual $K^\vee$ is a Hopf algebra equipped with a modular pair in involution
$(ev_\sigma, \Character )$
in the sense
of Connes-Moscovici where $ev_\sigma:K^\vee\rightarrow \mathbb{F}$ is defined by
$ev_\sigma(\varphi)=\varphi(\sigma)$.

Let $\psi_n:(K^\vee)^{\otimes n}\buildrel{\cong}\over\rightarrow (K^{\otimes n})^{\vee}$
be the linear map mapping the tensor product $\varphi_1\otimes\dots\otimes \varphi_n$
of $n$ forms on $K$ to the form on $K^{\otimes n}$, also denoted
$\varphi_1\otimes\dots\otimes \varphi_n$, mapping $k_1\otimes\dots\otimes k_n$
to the product $\varphi_1(k_1)\dots \varphi_n(k_n)$.
Then ii) $\psi_*$ is an isomorphism of cocyclic modules between the cocyclic
modules $\Omega (K^\vee)_{(ev_\sigma, \Character)}$ introduced by Connes-Moscovici
and the dual of the cyclic module $B(K)^{(\Character,\sigma)}$
introduced by Khalkhali-Rangipour~\cite[Theorem 2.1]{Khalkhali-Rangipour:newcyclic}
and Taillefer~\cite{Taillefer:cyclicHopf}.

iii) In particular, $\psi_*$ induces an isomorphism between
Connes-Moscovici cyclic cohomology of $K^\vee$, $HC^*_{(ev_\sigma,\Character)}(K^\vee)$
and the dual of Khalkhali-Rangipour-Taillefer cyclic homology of $K$,
$\widetilde{HC}_*^{(\Character,\sigma)}(K)^\vee$.
\end{proposition}
The cocyclic module $\Omega (H)_{(\Character,\sigma)}$ is denoted
$H_{(\Character,\sigma)}^{\natural} $ 
in~\cite[Theorem 1]{Connes-Moscovici:symmetry}.
The cyclic module $B(K)^{(\Character,\sigma)}$ is denoted $\widetilde{K}^{(\Character,\sigma)}$ in~\cite[Theorem 2.1]{Khalkhali-Rangipour:newcyclic} and
$C^{(\sigma,\varepsilon,\Character)}_*(K)$ in~\cite[2.1]{Taillefer:cyclicHopf}.
\begin{definition}\label{definition modular pair}
A {\it modular pair} is a couple $(\Character,\sigma)$ when $\Character$ is
a character and $\sigma$ is a group-like element such
that $\Character(\sigma)=1$.
\end{definition}
\begin{proof}[Proof of Proposition~\ref{dualite Connes-Moscovici Khalkhali-Rangipour}]
i) An element $\sigma$ is a {\it group-like} element of $K$ by definition if and only if
$\Delta \sigma=\sigma\otimes\sigma$ and $\varepsilon(\sigma)=1$. This means
that the linear map that we denoted again $\sigma:\mathbb{F}\rightarrow K$, mapping
$1$ to $\sigma$ is a morphism of coalgebras. Therefore its dual $ev_\sigma=\sigma^\vee:K^\vee\rightarrow \mathbb{F}$ is a morphism of algebras, i. e.
a {\it character} of $K^\vee$.
Let $\Character:K\rightarrow \mathbb{F}$ be a character of $K$, i. e.
a morphism of algebras. Its dual $\Character^\vee:\mathbb{F}\rightarrow K^\vee$,
 mapping $1$ to $\Character$, is a morphism of coalgebras, i.e $\Character$ is a group-like element of $K^\vee$.
By definition, $ev_\sigma(\Character)=\Character(\sigma)$. Therefore $(\Character,\sigma)$ is a modular pair on $K$ if and only if $(ev_\sigma,\Character)$
is a modular pair in $K^\vee$.

Let $(\Character,\sigma)$ be a modular pair on $H$.
The {\it twisted antipode} $\widetilde{S}$ associated to $(\Character,\sigma)$ (in the sense of Connes-Moscovici)  is by definition the convolution
product $(\eta\circ\Character)\star S$ in $\text{Hom}(H,H)$.
Explicitly, for $h\in H$, $\widetilde{S}(h)=\Character(h^{1})S(h^{2})$,
where $\Delta h=h^{1}\otimes h^{2}$.
Consider the map $\tau_n:H^{\otimes n}\rightarrow H^{\otimes n}$ defined by
$$
\tau_n(h_1\otimes\dots\otimes h_n):=
\mu_{H^{\otimes n}}\left(\Delta^{n-1}\widetilde{S}(h_1)\otimes
(h_2\otimes\cdots\otimes h_n\otimes \sigma)\right)
$$
Here $\mu_{H^{\otimes n}}:H^{\otimes n}\otimes H^{\otimes n}\rightarrow H^{\otimes n}$ is the product in $H^{\otimes n}$
and $\Delta^{n-1}:H\rightarrow H^{\otimes n}$ is the iterated diagonal on $H$.
In particular, $\tau_1(h)=\widetilde{S}(h)\sigma$.
\begin{definition}\label{definition modular pair in involution}
A modular pair $(\Character,\sigma)$
(Definition~\ref{definition modular pair}) is
{\it in involution} in the sense of Connes-Moscovici
if and only if $\tau_1^2=id_H$, i. e. $\forall h\in H$,
$\widetilde{S}^2(h)=\sigma h\sigma^{-1}$.
\end{definition}
Let $(\Omega H)_{(\Character,\sigma)}$ be the usual cosimplicial module
defining the Cobar construction, except that $\delta_{n+1}:H^{\otimes n}
\rightarrow H^{\otimes n+1}$ is given by~\cite[(2.9)]{Connes-Moscovici:symmetry}
 $\delta_{n+1}(h_1\otimes\dots\otimes h_n)=h_1\otimes\dots\otimes h_n\otimes\sigma$. Connes and Moscovici have shown that if $\tau_1^2=id_H$,
then  $(\Omega H)_{(\Character,\sigma)}$ (equipped with the $\tau_n$)
is a cocyclic module.

Let $(\Character,\sigma)$ be a modular pair on $K$.
Let $t_n:K^{\otimes n}\rightarrow K^{\otimes n}$ defined by
(~\cite[Theorem 2.1]{Khalkhali-Rangipour:newcyclic}
or~\cite[2.1]{Taillefer:cyclicHopf} which generalizes~\cite[(7.3.3.1)]{LodayJ.:cych})
$$t_n(k_1\otimes\dots\otimes k_n)=\sigma S(k_1^{(1)}\dots k_n^{(1)})
\otimes k_1^{(2)}\otimes\dots\otimes k_{n-1}^{(2)}\Character(k_n^{(2)})
$$
where $\Delta(k_i)=k_i^{(1)}\otimes k_i^{(2)}$.
In particular, $t_1$ is equal to $\sigma(S\star \eta\circ\Character)$,
the left multiplication by $\sigma$ of the convolution product
$\star$ of $S$ and $\eta\circ\Character$.
By definition, the couple $(\Character,\sigma)$ is a {\it modular pair
in involution} in the sense of Khalkhali-Rangipour~\cite[(1)]{Khalkhali-Rangipour:newcyclic} if and only if $t_1^2=id_K$.

Therefore to prove part i) of this Proposition, it suffices to show that
$\tau_1=t_1^\vee$. This will be proved in the proof of ii) below.

(Denote by $K^{op,cop}$ the Hopf algebra
with the opposite multiplication, the opposite diagonal and the same antipode
~\cite[Remark 4.2.10]{Dascalescu-Nastasescu-Raianu:hopfalgebras},
since the convolution product $\star$ on $\text{Hom}(K^{op,cop},K^{op,cop})$
is the opposite of the convolution product $\star$ on $\text{Hom}(K,K)$,
note that a modular pair in involution for $K$ in the sense
of Khalkhali-Rangipour is the same as a modular pair in involution
for $K^{op,cop}$ in the sense of Connes-Moscovici.)

ii) Let $B(K)^{(\Character,\sigma)}$ be the usual simplicial module
defining the Bar construction except that
$d_{n+1}:K^{\otimes n+1}\rightarrow K^{\otimes n}$
is given by (~\cite[Theorem 2.1]{Khalkhali-Rangipour:newcyclic}
or~\cite[2.1]{Taillefer:cyclicHopf})
$d_{n+1}(k_1\otimes\dots\otimes k_{n+1})=k_1\otimes\dots\otimes k_{n}\Character(k_{n+1})
$. It is well known~\cite[Lemma XVIII.7.3]{Kassel:quantumgrps} that
$\psi_*$ is an isomorphism of cosimplicial modules from the usual Cobar construction on $K^\vee$, $\Omega(K^\vee)_{(ev_\sigma,\varepsilon)}$, to the dual of the usual Bar construction on $K$,
$(B(K)^{(\varepsilon, \sigma)})^\vee$.
Obviously, $\psi_{n+1}\circ \delta_{n+1}=d_{n+1}^\vee\circ\psi_{n}$.
Therefore $\psi_*:\Omega(K^\vee)_{(ev_\sigma,\Character)}\buildrel{\cong}\over\rightarrow (B(K)^{(\Character, \sigma)})^\vee$ is an isomorphism of cosimplicial modules even if $\Character\neq \varepsilon$.

Denote by $\sigma S:K\rightarrow K$ the linear map defined by $(\sigma S)(k)=\sigma S(k)$, $k \in K$.
The transposition map $\text{Hom}(K,K)\rightarrow \text{Hom}(K^\vee,K^\vee)$,
mapping a linear map $f:K\rightarrow K$ to its dual  $f^\vee:K^\vee\rightarrow
K^\vee$ is a morphism of algebras with respect to the convolution products $\star$. Since  $\sigma S$ can be written as the convolution product
$(\sigma\circ\varepsilon) \star S$ of the composite
$K\buildrel{\varepsilon}\over\rightarrow \mathbb{F}\buildrel{\sigma}\over\rightarrow K$ and of the antipode $S$, its dual  $(\sigma S)^\vee$ is equal to
$(\varepsilon^\vee\circ\sigma^\vee) \star S^\vee=(\varepsilon \circ ev_\sigma)\star S^\vee$ which is the twisted antipode $\widetilde{S}$ on $K^\vee$ associated
to the modular pair $(ev_\sigma,\Character)$.

The cyclic operator $t_n:K^{\otimes n}\rightarrow K^{\otimes n}$
can be written as the composite
$$
K^{\otimes n}\buildrel{\Delta_{K^{\otimes n}}}\over\rightarrow K^{\otimes n}\otimes K^{\otimes n}\buildrel{\mu^{(n-1)}\otimes K^{\otimes n}}\over\rightarrow K\otimes K^{\otimes n}\buildrel{\sigma S\otimes K^{\otimes n-1}\otimes \Character}\over\rightarrow
  K\otimes K^{\otimes n-1}\otimes \mathbb{F}.
$$
Here $\mu^{(n-1)}:K^{\otimes n}\rightarrow K$ is the iterated product on $K$
and $\Delta_{K^{\otimes n}}$ is the diagonal of $K^{\otimes n}$.
The cocyclic operator $\tau_n:H^{\otimes n}\rightarrow H^{\otimes n}$
can be written as the composite
$$
H\otimes H^{\otimes n-1}\otimes \mathbb{F}
\buildrel{\widetilde{S}\otimes H^{\otimes n-1}\otimes \sigma}\over\rightarrow
H\otimes H^{\otimes n}
\buildrel{\Delta^{(n-1)}\otimes H^{\otimes n}}\over\rightarrow H^{\otimes n}\otimes H^{\otimes n}\buildrel{\mu_{H^{\otimes n}}}\over\rightarrow H^{\otimes n}\
$$
Here $\Delta^{(n-1)}:H\rightarrow H^{\otimes n}$ is the iterated diagonal on $H$
and $\mu_{H^{\otimes n}}$ is the multiplication of $H^{\otimes n}$.
We saw that the twisted antipode $\widetilde{S}$ on $K^\vee$ associated to
$(ev_\sigma,\Character)$ was $(\sigma S)^\vee$, the dual of $\sigma S$.
Therefore $\psi_n\circ \tau_n=t_n^\vee\circ\phi_n$. In particular when $n=1$, since $\psi_1$ is the identity, $\tau_1=t_1^\vee$.
So finally, $\psi_*:\Omega(K^\vee)_{(ev_\sigma,\Character)}\buildrel{\cong}\over\rightarrow (B(K)^{(\Character, \sigma)})^\vee$ is an isomorphism of cocyclic modules.
\end{proof}
\begin{corollary}\label{dualite operad cyclic bar cobar}
Let $K$ be a finite dimensional
Hopf algebra equipped with a group-like element $\sigma$ such that $\forall k\in K$,
$S\circ S(k)=\sigma^{-1}k\sigma$.
Then
$\psi_*:\Omega(K^\vee)_{(ev_\sigma,\varepsilon)}\buildrel{\cong}\over\rightarrow (B(K)^{(\varepsilon, \sigma)})^\vee$
is an isomorphism of cyclic operads with multiplication.
In particular, $H^*(\psi_*):\text{Cotor}_{K^\vee}^*(\mathbb{F},\mathbb{F})\buildrel{\cong}\over\rightarrow\text{Ext}_K^*(\mathbb{F},\mathbb{F})$ is an isomorphism of Batalin-Vilkovisky algebras
and 
$\psi_*$ induces an isomorphism of graded Lie algebras
$HC^*_{(ev_\sigma,\varepsilon)}(K^\vee)\buildrel{\cong}\over\rightarrow
\widetilde{HC}^*_{(\varepsilon,\sigma)}(K)$.
\end{corollary}
\begin{proof}
The canonical injection of $K$ into its bidual $K^{\vee\vee}$,
$\nu:K\hookrightarrow K^{\vee\vee}$, is an isomorphism of bialgebras.
Let $C:=K^\vee$ be the dual bialgebra.
In the proof of Proposition~\ref{dualite entre inclusions Gerstenhaber},
we saw that $\phi:\Omega C\buildrel{\cong}\over\rightarrow B(C^\vee)^\vee$
is an isomorphism of linear operads with multiplication.
Therefore the composite
$
\Omega (K^\vee)\buildrel{\phi}\over\rightarrow B(K^{\vee\vee})^\vee
\buildrel{B(\nu)^\vee}\over\rightarrow B(K)^\vee$ is also
an isomorphism of linear operads with multiplication.
But this composite coincides with the isomorphism of cocyclic
modules $\psi_*:\Omega(K^\vee)_{(ev_\sigma,\varepsilon)}\buildrel{\cong}\over\rightarrow (B(K)^{(\varepsilon, \sigma)})^\vee$ given by part
ii) of Proposition~\ref{dualite Connes-Moscovici Khalkhali-Rangipour}.
\end{proof}
\begin{theorem}\label{ext BV algebre}
Let $K$ be a Hopf algebra equipped with a group-like element $\sigma$ such
that for all $k\in K$, $S^2(k)=\sigma^{-1}k\sigma$.
Let $t_n:K^{\otimes n}\rightarrow K^{\otimes n}$ be the linear map
defined by
$$t_n(k_1\otimes\dots\otimes k_n)=\sigma S(k_1^{(1)}\dots k_{n-1}^{(1)}k_n)\otimes
k_1^{(2)}\otimes\dots\otimes k_{n-1}^{(2)}.$$
The dual of the Bar construction on $K$, $B(K)^\vee$ is a cyclic operad
with multiplication.
In particular, the Gerstenhaber algebra given by
Theorem~\ref{ext sous algebre de Gerstenhaber de Hochschild},
$\text{Ext}_K^*({\Bbbk},{\Bbbk})$, is in fact a Batalin-Vilkovisky algebra
and the cyclic cohomology of $K$, $\widetilde{HC}^*_{(\varepsilon,\sigma)}(K)$
has a Lie bracket of degre $-1$.
\end{theorem}
\begin{proof}
Corollary~\ref{dualite operad cyclic bar cobar} explains in details
that this Theorem is the dual of
Corollaries~\ref{Cobar algebre de Hopf BV} and~\ref{crochet sur Hopf cyclic cohomologie}. Therefore, the computations dual to~\cite[Proof of Theorem 1.1]{MenichiL:BValgaccoHa} show that the operad with multiplication 
$B(K)^\vee$ given in the proof of
Theorem~\ref{ext sous algebre de Gerstenhaber de Hochschild}
together with the cyclic operators $t_n$ defines a cyclic operad with
multiplication.
Using Theorems~\ref{Theoreme BV algebre}
and~\ref{crochet sur cohomologie cyclique lambda}, we conclude.
\end{proof}
\section{Characteristic maps}
\begin{lemma}\label{morphisme de Gerstenhaber de cotor vers cohomologie de Hochschild}
Let $H$ be a bialgebra.
Let $A$ be a left module algebra over $H$ (in the sense of~\cite[Definition V.6.1]{Kassel:quantumgrps}).
Then the application $\Phi:H^{\otimes n}\rightarrow \text{Hom}_{{\Bbbk}-mod}(A^{\otimes n},A)$
mapping $h_1\otimes\dots\otimes h_n$ to $f:A^{\otimes n}\rightarrow A$ defined by
$f(a_1\otimes\dots\otimes a_n)=(h_1.a_1)\dots(h_n.a_n)$
defines a morphism of linear operads with multiplication from the coendomorphism operad of $H$,
$\mathcal{C}o\mathcal{E}nd_{H-mod}(H)$, to the endomorphism operad of $A$, $\mathcal{E}nd_{{\Bbbk}-mod}(A)$.
In particular, $\Phi$ induces a morphism of Gerstenhaber algebras
$H^*(\Phi):\text{Cotor}_H^*({\Bbbk},{\Bbbk})\rightarrow HH^*(A,A)$.
\end{lemma}
\begin{proof}%[First proof]
Since $1_H.a_1=a_1$, $\Phi(1_H)=id_A$.
Let $h_1\otimes\dots\otimes h_m\in H^{\otimes m}$, $k_1\otimes\dots\otimes k_m\in H^{\otimes n}$ and
$a_1\otimes\dots\otimes a_{m+n-1}\in A^{\otimes m+n-1}$.
Using~(\ref{Composition operad cobar}), we have that
$\Phi[(h_1\otimes\dots\otimes h_m)\circ_i(k_1\otimes\dots\otimes k_n)]
$ evaluated on $a_1\otimes\dots\otimes a_{m+n-1}$ is equal to the product
$$(h_1. a_1)\dots (h_{i-1}. a_{i-1})(h_{i}^{(1)}k_1. a_{i})\dots
(h_{i}^{(n)}k_n. a_{i+n-1})(h_{i+1}. a_{i+n})\dots (h_m.a_{m+n-1}).
$$
On the other hand, using example~\ref{operade endomorphisme},
$\Phi(h_1\otimes\dots\otimes h_m)\circ_i\Phi(k_1\otimes\dots\otimes k_n)$
evaluated on $a_1\otimes\dots\otimes a_{m+n-1}$ is equal to the product
$$(h_1. a_1)\dots (h_{i-1}. a_{i-1})
\left(h_{i}.[(k_1. a_{i})\dots(k_n. a_{i+n-1})]\right)
(h_{i+1}. a_{i+n})\dots (h_m.a_{m+n-1}).
$$
Since for any $h\in H$, $a$ and $b\in A$,
$h.(ab)=(h^{(1)}.a)(h^{(2)}.a)$, the previous two products are equal.
So $\Phi$ is a morphism of operads.
Now $\Phi$ is a morphism of operads with multiplication,
since $\Phi(1_{\Bbbk})$ is the unit map $\eta:{\Bbbk}\rightarrow A$
and since $\Phi(1_H\otimes 1_H)$ is the multiplication $\mu:A\otimes A\rightarrow A$
of $A$.
\end{proof}
The following Lemma is a variant of the previous lemma if $H$ is finite dimensional,
since in this case, $A$ is a left module algebra over $H$
if and only if $A$ be a right comodule algebra over the dual of $H$, $H^\vee$.
\begin{lemma}\label{morphisme de Gerstenhaber de Ext vers cohomologie de Hochschild}
Let $H$ be a bialgebra.
Let $A$ be a right comodule algebra over $H$ (in the sense of~\cite[Definition III.7.1]{Kassel:quantumgrps}).
Then the application $\Phi:(H^{\otimes n})^\vee\rightarrow \text{Hom}_{{\Bbbk}-mod}(A^{\otimes n},A)$ mapping $f:H^{\otimes n}\rightarrow {\Bbbk}$
to $F:A^{\otimes n}\rightarrow A$ defined by
$F(a_1\otimes\dots\otimes a_n)=
a_1^{(1)} \dots a_n^{(1)}
f(a_1^{(2)}\otimes\dots\otimes a_n^{(2)}$).
defines a morphism of linear operads with multiplication from the endomorphism operad of $H$,
$\mathcal{E}nd_{H-comod}(H)$, to the endomorphism operad of $A$, $\mathcal{E}nd_{{\Bbbk}-mod}(A)$.
In particular, $\Phi$ induces a morphism of Gerstenhaber algebras
$H^*(\Phi):\text{Ext}_H^*({\Bbbk},{\Bbbk})\rightarrow HH^*(A,A)$.
\end{lemma}
Note that in the case $A=H$, $H^*(\Phi)$ coincides with the inclusion
of Gerstenhaber algebras given by
Theorem~\ref{ext sous algebre de Gerstenhaber de Hochschild}.
The proof of Lemma~\ref{morphisme de Gerstenhaber de Ext vers cohomologie de Hochschild} is a computation similar to the proof of Lemma~\ref{morphisme de Gerstenhaber de cotor vers cohomologie de Hochschild}.
\begin{theorem}\label{characteristic de Khalkhali-Rangipour morphisme de Lie}
Let $H$ be a Hopf algebra equipped with a group-like element $\sigma\in H$
such that $\forall h\in H$, $S^2(h)=\sigma^{-1}h\sigma$.
Let $A$ be a right comodule algebra over $H$.
Let $\tau:A\rightarrow {\Bbbk}$ be a non degenerate $1$-trace, i. e.
the morphism of left $A$-modules
$\Theta:A\buildrel{\cong}\over\rightarrow A^\vee$, mapping $b\in A$
to $\varphi:A\rightarrow {\Bbbk}$ given by
$\varphi(a)=\tau(ab)$, is an isomorphism of $A$-bimodules.
Suppose that $\tau$ is $\sigma$-invariant in the sense of~\cite[Definition 3.1]{Khalkhali-Rangipour:newcyclic}: $\forall a,b\in A$, $\tau(a^{(1)})a^{(2)}=\tau(a)\sigma$.
Then 

1) the morphism 
$H^*(\Phi):\text{Ext}_H^*({\Bbbk},{\Bbbk})\rightarrow HH^*(A,A)$
given by
Lemma~\ref{morphisme de Gerstenhaber de cotor vers cohomologie de Hochschild},
is a morphism of Batalin-Vilkovisky algebras,

2) the characteristic map defined by Khalkhali-Rangipour~\cite[(10)]{Khalkhali-Rangipour:newcyclic} $\gamma^*:\widetilde{HC}_{(\varepsilon,\sigma)}^*(H)\rightarrow HC^*_\lambda(A)$
is a morphism of graded Lie algebras.
\end{theorem}
\begin{proof}
Recall from the proof of
Corollary~\ref{Cohomologie Hochschild algebre symetrique BV} that
$$\mathcal{C}^*(A,A)\build\rightarrow_\cong^{\mathcal{C}^*(A,\Theta)}
\mathcal{C}^*(A,A^\vee)\build\rightarrow_\cong^{Ad}
\mathcal{C}_*(A,A)^\vee$$
is a cyclic operad with multiplication.
By Lemma~\ref{morphisme de Gerstenhaber de Ext vers cohomologie de Hochschild},
$\Phi:B(H)^\vee\rightarrow \mathcal{C}^*(A,A)$ is a morphism of linear operads with
multiplication. Let $\gamma:\mathcal{C}_*(A,A)\rightarrow B(H)$
be the morphism of cyclic modules defined
by~\cite[Proposition 3.1]{Khalkhali-Rangipour:newcyclic}
$$\gamma(a_0\otimes a_1\otimes\dots\otimes a_n)=\tau(a_0 a_1^{(1)}\dots a_n^{(1)})(a_1^{(2)}\otimes\dots\otimes a_n^{(2)}).$$
Here the coaction of $a_i$, $\Delta a_i=a_i^{(1)}\otimes a_i^{(2)}$.
A straightforward calculation shows that the composite
$$B(H)^\vee\buildrel{\Phi}\over\rightarrow\mathcal{C}^*(A,A)\build\rightarrow_\cong^{\mathcal{C}^*(A,\Theta)}
\mathcal{C}^*(A,A^\vee)\build\rightarrow_\cong^{Ad}
\mathcal{C}_*(A,A)^\vee$$
is the dual of $\gamma$, $\gamma^\vee:\mathcal{C}_*(A,A)^\vee\rightarrow B(H)^\vee$. Since $\gamma^\vee$ is a morphism of cocyclic modules,
$\Phi:B(H)^\vee\rightarrow \mathcal{C}^*(A,A)$ is a morphism of linear cyclic operads with multiplication. By applying Theorem~\ref{Theoreme BV algebre},
$H(\Phi)$ is a morphism of Batalin-Vilkovisky algebras between the
Batalin-Vilkovisky algebras given by Theorem~\ref{ext BV algebre}
and Corollary~\ref{Cohomologie Hochschild algebre symetrique BV}. This is 1).
By applying Theorem~\ref{crochet sur cohomologie cyclique lambda}, 
we obtain 2).
\end{proof}
Using this time, Lemma~\ref{morphisme de Gerstenhaber de cotor vers cohomologie de Hochschild} and the cocyclic map $\gamma$ defined by Connes and
Moscovici~\cite[Theorem 6]{Connes-Moscovici:symmetry},
we obtain easily the following variant of the previous Theorem.
\begin{theorem}\label{characteristic de Connes-Moscovici morphisme de Lie}
Let $H$ be a Hopf algebra endowed with a modular pair in
involution $(\Character,1)$. Let $A$ be a module algebra over $H$.
Let $\tau:A\rightarrow {\Bbbk}$ be a non degenerate $1$-trace, i. e.
the morphism of left $A$-modules
$\Theta:A\buildrel{\cong}\over\rightarrow A^\vee$, mapping $b\in A$
to $\varphi:A\rightarrow {\Bbbk}$ given by
$\varphi(a)=\tau(ab)$, is an isomorphism of $A$-bimodules.
Suppose that $\tau$ is $\Character$-invariant.
Then 

1) the morphism 
$H^*(\Phi):\text{Cotor}_H^*({\Bbbk},{\Bbbk})\rightarrow HH^*(A,A)$
given by
Lemma~\ref{morphisme de Gerstenhaber de cotor vers cohomologie de Hochschild},
is a morphism of Batalin-Vilkovisky algebras,

2) the characteristic map defined by Connes and Moscovici
$\chi_{\tau}:HC_{(\Character,1)}^*(H)\rightarrow HC^*_\lambda(A)$
is a morphism of graded Lie algebras.
\end{theorem}
In~\cite[Theorem 6]{Connes-Moscovici:symmetry} or\cite[Section 4.4]{Skandalis:geononcom}, Connes and Moscovici have defined more generally a characteristic
map $\chi_{\tau}:HC_{(\Character,\sigma)}^*(H)\rightarrow HC^*_\lambda(A)$
without assuming that

i) the group-like element $\sigma$ is the unit $1$ of $H$,

\noindent and without assuming that

ii) the $\sigma$-trace $\tau$ is non degenerated.

But we need i) to have a Lie bracket on $HC_{(\Character,\sigma)}^*(H)$
(Corollary~\ref{crochet sur Hopf cyclic cohomologie})
and we need i) and ii) to have a Lie bracket on  $HC^*_\lambda(A)$
(Corollary~\ref{crochet sur la Cohomologie cyclic rationelle}).
However, note that in their first construction of the characteristic map
in~\cite{Connes-Mosco:hopfacctit}, Connes and Moscovici were assuming i) like us. We believe that ii) can be weakened, since the Batalin-Vilkovisky
algebra on $HH^*(A,A^\vee)$ can be defined for non counital symmetric
Frobenius algebras, i. e ``unital associative algebras with an invariant
co-inner product''~\cite[p. 61-2]{Tradler-Zeinalian:algebraicstringoperations}.
In particular, as Tradler explained us, $A$ does not need to be finite
dimensional.
\section{Hopf algebras that are symmetric Frobenius}\label{algebre de Hopf symmetrique}
In this section, we work over an arbitrary field $\mathbb{F}$.
We want to consider in Theorem~\ref{characteristic de Khalkhali-Rangipour morphisme de Lie}, the case
 where the the comodule algebra $A$ over $H$ is the Hopf algebra $H$ itself.
We remark that for a finite dimensional Hopf algebra $H$, there is a close relationship between being
a symmetric Frobenius algebra and being equipped with a modular pair in involution of the form $(\varepsilon,u)$
(Theorem~\ref{symmetrique ssi unimodulaire}). Therefore
(Theorem~\ref{inclusion BV ext hochschild}), for Hopf algebras which are symmetric Frobenius algebras, often we have an inclusion of Batalin-Vilkovisky
algebras $\text{Ext}^*_H(\mathbb{F},\mathbb{F})\hookrightarrow 
HH^*(H,H)$ and in some cases, the characteristic map
$\widetilde{HC}_{(\varepsilon,\sigma)}^*(H)\rightarrow HC^*_\lambda(H)$
is injective.

First, we recall the notion of (symmetric) Frobenius algebra
and that the Nakayama automorphisms of a symmetric Frobenius algebra
are all inner automorphisms. Then we recall that an augmented symmetric
Frobenius algebra is always unimodular.
Specializing to Hopf algebras, we recall that 
finite dimensional Hopf
algebras are always Frobenius algebras and that
the square $S\circ S$
of the antipode of an unimodular Hopf algebra
is a particular Nakayama automorphism.
Finally, we can recall Theorem~\ref{symmetrique ssi unimodulaire}
due to Oberst and Schneider~\cite{Oberst-Schneider},
which explains when a Hopf algebra is a symmetric Frobenius algebra.
In the proof of Theorem~\ref{symmetrique ssi unimodulaire},
we recall the construction of a non-degenerated trace $\tau$ on $H$.
Checking that the diagonal of $H$ is compatible with this trace $\tau$,
we obtain Theorem~\ref{inclusion BV ext hochschild}.
\subsection{Frobenius algebras}
Let $A$ be an algebra.
The morphism of right $A$-modules $\Theta:A\rightarrow A^\vee$,
mapping $1$ to the form $\phi$, is an isomorphism (of $A$-bimodules)
if and only if $A$ is finite dimensional and the bilinear form
$<\quad,\quad >:A\otimes A\rightarrow \mathbb{F}$ defined
by $<a,b>:=\phi(ab)$ is non degenerate (and symmetric).
\begin{definition}\label{def algebre de Frobenius}
An algebra $A$ is a {\it (symmetric) Frobenius algebra}
if there exists an isomorphism  $\Theta:A\buildrel{\cong}\over\rightarrow A^\vee$ of right $A$-modules (respectively of $A$-bimodules).
We call $\phi:=\Theta(1)$ a {\it Frobenius form}.
\end{definition}
\begin{example}
Let $G$ be a finite group then its group algebra $\mathbb{F}[G]$ is a non commutative symmetric Frobenius algebra. By definition, the group ring $\mathbb{F}[G]$ admits the set $\{g\in G\}$ as a basis.
Denote by $\delta_g$ the dual basis in $\mathbb F[G]^\vee$.
The linear isomorphism $\Theta:\mathbb{F}[G]\rightarrow \mathbb{F}[G]^\vee$,
sending $g$ to $\delta_{g^{-1}}$ is an isomorphism of $\mathbb{F}[G]$-bimodules.
\end{example}
Let $A$ be a Frobenius algebra with Frobenius form $\phi$.
By definition~\cite[(16.42)]{Lam:lecturemodulesrings}, the {\it Nakayama automorphism}
of $\phi$ is the unique automorphism of algebras
$\sigma:A\buildrel{\cong}\over\rightarrow A$ such that
for all $a$ and $b\in A$, $\phi(ab)=\phi(\sigma(b)a)$.
Let $\sigma$ and $\sigma'$ be two Nakayama automorphisms of a Frobenius
algebra $A$. Then, by~\cite[(16.43)]{Lam:lecturemodulesrings}, there exists
an invertible element $u\in A$ such that for all $x\in A$,
$\sigma'(x)=u\sigma(x)u^{-1}$. In particular, if $A$ is a symmetric
Frobenius algebra, the identity map of $A$, $id_A:A\rightarrow A$ 
is a particular Nakayama automorphism of $A$. And all the other
Nakayama automorphisms are inner
automorphisms~\cite[p. 483 Lemma (b)]{Lorenz:representationHopf}.
\begin{definition}\label{definition integrale unimodular}
Let $(A,\mu,\eta,\varepsilon)$ be an augmented algebra.
A left (respectively right) {\it integral} in $A$ is an element $l$ of $A$ such
that $\forall h\in A$, $h\times l=\varepsilon (h) l$ (respectively $l\times h=\varepsilon(h) l$).
The augmented algebra $A$ is {\it unimodular} if the set of left integrals
in $A$ coincides with the set of right integrals in $A$.
\end{definition}
\begin{remark}\label{Haar integral bar acyclic}
An element $l$ of $A$ is a right integral in $A$ such that $\varepsilon(l)=1$
if and only if $1_A-l$ is a left unit in $\overline{A}$, the augmentation ideal of $A$.
Suppose that there exists a  right integral $l$ in $A$ such that $\varepsilon(l)=1$.
Then $l$ defines a morphism of right $A$-modules $s:\mathbb{F}\rightarrow A$
such that $\varepsilon \circ s=id_\mathbb{F}$. Therefore $\mathbb{F}$ is a right projective
$A$-module and $\text{Ext}^*_A(\mathbb{F},\mathbb{F})$ is concentrated in degre $0$
(Compare with~\cite[Proof of Proposition 5.4]{Crainic:cylcohhopfalg}).
\end{remark}
The set of right (respectively left) integrals in an augmented Frobenius algebra
is a $\mathbb{F}$-vector space of dimension $1$~\cite[Proposition 6.1]{Kadison:Frobeniusextension}.
Let $A$ be an augmented algebra and let $\Theta:A\buildrel{\cong}\over\rightarrow A^\vee$
be an isomorphism of left (respectively right) $A$-modules.
Then $\Theta^{-1}(\varepsilon)$ is non-zero left (respectively right)
integral in $A$~\cite[just above Proposition 6.1]{Kadison:Frobeniusextension}.
In particular, if $\Theta:A\buildrel{\cong}\over\rightarrow A^\vee$ is an isomorphism of $A$-bimodules, $\Theta^{-1}(\varepsilon)$ is both a non-zero left et right
integral in $A$. Therefore a symmetric Frobenius algebra with an augmentation
is always unimodular.

Let $A$ be a Frobenius algebra with an augmentation. Let $t$ be any non-zero
left integral in $A$.
The {\it distinguished group-like element} or left modular function
~\cite[(6.2)]{Kadison:Frobeniusextension} in $A^\vee$ is the unique morphism
of algebras $\alpha:A\rightarrow\mathbb{F}$ such that for all $h\in
A$,  $t\times h=\alpha(h) t$ (\cite[2.2.3]{Montgomery:Hopfalgactring} or~\cite[p. 590]{Radford:traceHopf}).
Note that $A$ is unimodular if and only if the distinguished group-like element in $A^\vee$ 
is $\varepsilon$ the augmentation of $A$.
\subsection{Hopf algebras}
Let $(H,\mu,\eta,\Delta,\varepsilon,S)$ be a finite dimensional Hopf algebra.
Its dual is also a Hopf algebra $(H^\vee, \Delta^\vee, \varepsilon^\vee,
\mu^\vee,\eta^\vee,S^\vee)$.
In particular, a form $\lambda$ on $H$ is 
a left (respectively right) integral in $H^\vee$ if and only if
for every $\varphi\in H^\vee$ and $k\in H$,
$\sum \varphi(k^{(1)})\lambda(k^{(2)})=\varphi(1_H)\lambda(k)$
(respectively 
$\sum \lambda(k^{(1)})\varphi(k^{(2)})=\varphi(1_H)\lambda (k)$).
Here $\Delta k=\sum k^{(1)}\otimes k^{(2)}$.
\begin{example}\label{exemples d'algebres de Frobenius}\cite[2.1.2]{Montgomery:Hopfalgactring}
If $G$ is a finite group,
$\sum_{g\in G} g$ is both a left and right integral in the group algebra
${\mathbb F}[G]$. And $\delta_1$, the form mapping $g\in G$ to $1$ if $g=1$
and $0$ otherwise, is both a left and right integral in ${\mathbb F}[G]^\vee$.
Since $\delta_1(1)=1$, by Remark~\ref{Haar integral bar acyclic},
$\text{Cotor}_{{\mathbb F}[G]}^*(\mathbb{F},\mathbb{F})$ and
$\text{Ext}_{{\mathbb F}[G]^\vee}^*(\mathbb{F},\mathbb{F})$ are both concentrated in degre $0$
(Note that here the product on $G$ is not used and that $G$ does not need to be finite
~\cite[4. p. 97]{Khalkhali:lecturesnoncommgeom}).
\end{example}
The set of left (respectively right)
integrals in the dual Hopf algebra $H^\vee$
is a ${\mathbb F}$-vector space of dimension $1$~\cite[Corollary 5.1.6 2)]{Sweedler:livre}.
So let $\lambda$ be any non-zero left (respectively right) integral in $H^\vee$.
The morphism of left (respectively right) $H$-modules, $H\buildrel{\cong}\over\rightarrow H^{\vee}$,
sending $g$ to the form, denoted~\cite[p. 95]{Sweedler:livre} $g\rightharpoonup\lambda$, mapping $h$ to $\lambda (hg)$
(respectively to the form mapping $h$ to $\lambda(gh)$), is an
isomorphism~\cite[Proof of Corollary 5.1.6 2)]{Sweedler:livre}.
So a finite dimensional Hopf algebra is always a Frobenius algebra,
but not always a symmetric Frobenius algebra as Theorem~\ref{symmetrique ssi unimodulaire} will show.
\begin{lemma}\label{nakayama right integral}(\cite[Theorem 3(a)]{Radford:traceHopf}
or \cite[(6.8)]{Kadison:Frobeniusextension})
Let $H$ be a finite dimensional Hopf algebra.
Let $\lambda$ be a non-zero right integral in $H^\vee$.
Let $\alpha$ be the distinguished group-like element in $H^\vee$.
Then for all $a$ and $b\in H$,

\noindent i) $
\lambda(ab)=\lambda(S^2(b\leftharpoonup \alpha)a)
$
where $b\leftharpoonup \alpha=\sum \alpha(b^{(1)})b^{(2)}$(\cite[p. 95]{Sweedler:livre}, \cite[p. 585]{Radford:traceHopf}
or \cite[p. 55]{Kadison:Frobeniusextension}),

\noindent ii) In the case $H$ is unimodular, $
\lambda(ab)=\lambda(S^2(b)a)
$~\cite[Proposition 8]{Larson-Sweedler:assorthobil}.
\end{lemma}
We have seen that if $H$ is unimodular, then $ \alpha =\varepsilon $.
Therefore ii) follows from i).
Note that Kadison's distinguished group-like element~\cite[(6.2) or p. 57]{Kadison:Frobeniusextension}
$m$ in $H^\vee$ is $\alpha^{-1}=\alpha\circ S$, the inverse of ours
(\cite[2.2.3]{Montgomery:Hopfalgactring} or\cite[p. 57]{Kadison:Frobeniusextension}),
since he uses right integrals to define it and we use left integrals.
Lemma~\ref{nakayama right integral} means that the Nakayama automorphism $\sigma$ of any non-zero right
integral $\lambda$ in $H^\vee$ is given by $\sigma(b)=S^2(b\leftharpoonup \alpha)$
for any $b\in H$.
%\begin{theorem}(Due to~\cite{Oberst-Schneider}. Other proofs are given
%in ~\cite{Farnsteiner} and~\cite[p. 487 Proposition]{Lorenz:representationHopf}.
%See also~\cite{Humphreys:symHopf}.)\label{symmetrique ssi unimodulaire}
\begin{theorem}~\cite{Oberst-Schneider,Farnsteiner,Lorenz:representationHopf,Humphreys:symHopf}\label{symmetrique ssi unimodulaire}
A finite dimensional Hopf algebra $H$ is a symmetric Frobenius algebra if and only if
$H$ is unimodular and its antipode $S$ satisfies $S^2$ is an inner
automorphism of $H$.
\end{theorem}
\begin{proof}
Suppose that $H$ is a symmetric Frobenius algebra. Then we saw that $H$ is unimodular
and that all its Nakayama automorphisms are inner automorphisms.
By ii) of Lemma~\ref{nakayama right integral}, $S^2$ is a Nakayama automorphism of $H$.

Conversely, assume that $H$ is unimodular and that
$S^2$ is an inner automorphism of $H$.
Let $u$ be an invertible element of $H$
such that
$\forall h\in H$, $S^2(h)=uhu^{-1}$.
Let $\lambda$ be any non-zero right integral in $H^\vee$.
We saw above that $\lambda(ab)$ is a non-degenerate bilinear form on $H$.
By ii) of Lemma~\ref{nakayama right integral}, 
 $
\lambda(ab)=\lambda(S^2(b)a)=\lambda(ubu^{-1}a)
$.
Therefore $\beta(h,k):=\lambda (uhk)$ is a non-degenerate symmetric
bilinear form~\cite[p. 487 proof of Proposition]{Lorenz:representationHopf}.
\end{proof}
\begin{example} Let $G$ be a finite group.
Since $S^2=Id$ and since $\delta_1$ is a right integral for
${\mathbb F}[G]^\vee$, we recover that the
the linear isomorphism $\mathbb{F}[G]\rightarrow\mathbb{F}[G]^\vee$, sending $g$ to
$\delta_1(g-)=\delta_{g^{-1}}$ is an isomorphism of $\mathbb{F}[G]$-bimodules.
\end{example}
The Sweedler algebra is an example of non unimodular Hopf algebra over
any field of characteristic different from $2$~\cite[2.1.2]{Montgomery:Hopfalgactring}.
Notice that a cocommutative Hopf algebra over a field of characteristic different from $0$
can be non unimodular~\cite[p. 487-8, Remark and Examples (1) and (4)]{Lorenz:representationHopf}.

The square of the antipode of every quasi-cocommutative Hopf algebra with bijective antipode
is an inner automorphism
(\cite[Proposition VIII.4.1]{Kassel:quantumgrps} or~\cite[10.1.4]{Montgomery:Hopfalgactring}).
Therefore by Theorem~\ref{symmetrique ssi unimodulaire}, every braided (also called
quasitriangular) unimodular finite dimensional Hopf algebra
is a symmetric Frobenius algebra. In particular, the Drinfeld double $D(H)$ of any
finite dimensional Hopf algebra is a symmetric Frobenius algebra~\cite[Theorem 6.10]{Kadison:Frobeniusextension}.
\begin{theorem}\label{inclusion BV ext hochschild}
Let $H$ be a finite dimensional unimodular (Definition~\ref{definition integrale unimodular})
Hopf algebra equipped with a group-like element $\sigma$ such that $\forall h\in H$,
$S^2(h)=\sigma^{-1}h\sigma$.
Then 1) $H^*(\Phi):\text{Ext}_H^*({\mathbb{F}},{\mathbb{F}})\hookrightarrow HH^*(H,H)$
is an inclusion of Batalin-Vilkovisky algebras.

2) Suppose moreover that $H$ is cosemisimple.
Then $\gamma^*:\widetilde{HC}_{(\varepsilon,\sigma)}^*(H)\rightarrow HC^*_\lambda(H)$
is an inclusion of graded Lie algebras.
\end{theorem}
Remark that by~\cite[Exercice 5.5.10]{Dascalescu-Nastasescu-Raianu:hopfalgebras},
a finite dimensional cosemisimple Hopf algebra is always unimodular.
\begin{proof}
By Theorem~\ref{ext sous algebre de Gerstenhaber de Hochschild},
$H^*(\Phi):\text{Ext}^*_H(\mathbb{F},\mathbb{F})\hookrightarrow
HH^*(H,H)$ is an inclusion of Gerstenhaber algebras.
By Theorem~\ref{ext BV algebre}
(or Corollary~\ref{dualite operad cyclic bar cobar}),
$\text{Ext}^*_H(\mathbb{F},\mathbb{F})$ is a Batalin-Vilkovisky algebra
and $\widetilde{HC}_{(\varepsilon,\sigma)}^*(H)$ has a Lie bracket of degre $-1$.
By Theorem~\ref{symmetrique ssi unimodulaire}, $H$ is a symmetric Frobenius
algebra. Therefore by
Corollary~\ref{Cohomologie Hochschild algebre symetrique BV},
$HH^*(H,H)$ is a Batalin-Vilkovisky algebra.
And by Corollary~\ref{crochet sur la Cohomologie cyclic rationelle},
$HC^*_\lambda(H)$ has a Lie bracket of degree $-1$.

More precisely, let $\lambda$ be any non-zero right integral in $H^\vee$.
Let $\tau:H\rightarrow \mathbb{F}$ given by $\tau(a)=\lambda(\sigma^{-1}a)$
for all $a\in H$. In the proof of Theorem~\ref{symmetrique ssi unimodulaire},
we saw that $\tau$ is a non degenerate $1$-trace.
Since $\lambda$ is right integral in $H^\vee$, using the canonical
injection of $H$ into its bidual, for every $k\in H$, $ \lambda(k^{(1)})k^{(2)}=\lambda (k)1_H$.
Here $\Delta k=k^{(1)}\otimes k^{(2)}$.
By taking $k=\sigma^{-1}a$, since $\sigma^{-1}$ is a group-like element,
for all $a\in H$, 
$\tau(a^{(1)})\sigma^{-1}a^{(2)}=\lambda(\sigma^{-1}a^{(1)})\sigma^{-1}a^{(2)}
=\lambda(\sigma^{-1}a)1_H=\tau(a)1_H$. This means that $\tau$ is
$\sigma$-invariant in the sense of~\cite[Definition 3.1]{Khalkhali-Rangipour:newcyclic}. Therefore by applying part 1) of
Theorem~\ref{characteristic de Khalkhali-Rangipour morphisme de Lie}
in the case $A=H$, we obtain that 
$H^*(\Phi):\text{Ext}_H({\mathbb{F}},{\mathbb{F}})\hookrightarrow HH^*(H,H)$
is a morphism of Batalin-Vilkovisky algebras. This is 1).

By~\cite[2.4.6]{Montgomery:Hopfalgactring} or~\cite[Exercice 5.5.9]{Dascalescu-Nastasescu-Raianu:hopfalgebras}, $H$ is cosemisimple means that there exists a right integral
$t$ in $H^\vee$ such that $t (1)=1$.
Since the set of right integrals in $H^\vee$ is a $\mathbb{F}$-vector space
of dimension $1$,
any non-zero right integral $\lambda$ in $H^\vee$ satisfies
$\lambda(1)\neq 0$.
Since $\tau(\sigma)=\lambda(\sigma^{-1}\sigma)=\lambda(1)$ is different
from zero, by~\cite[Theorem 3.1]{Khalkhali-Rangipour:newcyclic},
the morphism of graded Lie algebras given by part 2) of
Theorem~\ref{characteristic de Khalkhali-Rangipour morphisme de Lie},
$\gamma^*:\widetilde{HC}_{(\varepsilon,\sigma)}^*(H)\rightarrow HC^*_\lambda(H)$
is injective. So 2) is proved.
\end{proof}
Note that by Theorem~\ref{symmetrique ssi unimodulaire}, any Hopf algebra
satisfying the hypotheses of Theorem~\ref{inclusion BV ext hochschild}
is a symmetric Frobenius algebra.
On the contrary, any Hopf algebra which is also a symmetric
Frobenius algebra does not necessarily satisfies the hypotheses of Theorem~\ref{inclusion BV ext hochschild}. Indeed, in a symmetric Frobenius Hopf algebra,
$S^2$ is an inner automorphism, not necessarily given by a group-like
element $\sigma$. But in order, to apply Connes-Moscovici (or more precisely
its dual Khalkhali-Rangipour-Taillefer) Hopf cyclic cohomology,
we have to suppose that $\sigma$ is a group-like element.

\section{The Batalin-Vilkovisky algebra $\text{Cotor}^*_{UL}(\mathbb{Q},\mathbb{Q})$}
Let $A$ be a Gerstenhaber algebra. Then $A^1$ is a Lie algebra.
This forgetful functor from Gerstenhaber algebras to Lie algebras,
has a left adjoint (\cite[Theorem 5
p. 67]{Gerstenhaber-Schack:algbqgad}
or~\cite[beginning of Section 4]{Gaudens-Menichi}), namely $L\mapsto \Lambda^* L$
where $\Lambda^* L$ is the exterior algebra on the Lie algebra $L$ equipped
with the {\it Schouten bracket}:
for $x_1\wedge\dots\wedge x_p\in\Lambda^p L$ and
$y_1\wedge\dots\wedge y_q\in\Lambda^q L$,
\begin{multline}\label{Schouten bracket}
\{x_1\wedge\dots\wedge x_p,y_1\wedge\dots\wedge y_q\}=\\
\sum_{1\leq i\leq p,1\leq j\leq q}\pm \{x_i,y_j\}\wedge x_1\wedge\dots\wedge\widehat{x_i}\wedge\dots\wedge
x_p\wedge y_1\wedge\dots\wedge\widehat{y_j}\wedge\dots\wedge y_q.
\end{multline}
Here the symbol $\widehat{\;}$ denotes omission and $\pm$ is the sign
$(-1)^{i+j+(p+1)(q-1)}$.
A tedious calculation shows that more generally in any Gerstenhaber algebra $A$, for $x_1,\dots,x_p, y_1,\dots,y_q\in A$,
\begin{multline*}
\{x_1\dots x_p,y_1\dots y_q\}=\\
\sum_{1\leq i\leq p,1\leq j\leq q}\pm \{x_i,y_j\} x_1\dots\widehat{x_i}\dots
x_p y_1\dots\wedge\widehat{y_j}\dots y_q.
\end{multline*}
where here $\pm$ is the sign
$(-1)^{\vert x_i\vert \vert x_1\dots x_{i-1}\vert+\vert y_j\vert \vert y_1\dots y_{j-1}\vert+(\vert x_1\dots x_p\vert+1)\vert y_1\dots\widehat{y_j}\dots y_q\vert }$.

In particular for any bialgebra $C$, the inclusion
of Lie algebras $P(C)\hookrightarrow\text{Cotor}^*_C({\Bbbk},{\Bbbk})$
given by Property~\ref{cotor en degre un egal primitif comme algebre
  de Lie}
induces an unique morphism of Gerstenhaber algebras
$\varphi_C:\Lambda^* P(C)\rightarrow \text{Cotor}^*_C({\Bbbk},{\Bbbk})$.  
\begin{proposition}~\cite[Theorem 8
  p. 70]{Gerstenhaber-Schack:algbqgad}\label{iso gerstenhaber cotor UL}
Let $L$ be a Lie algebra over the rationals $\mathbb{Q}$.
Consider the universal envelopping algebra $UL$ with its canonical
bialgebra
structure. Then the morphism of Gerstenhaber algebras
$\varphi_{UL}:\Lambda^* L\rightarrow
\text{Cotor}^*_{UL}(\mathbb{Q},\mathbb{Q})$
is an isomorphism.
\end{proposition}
Since we have not be able to fully understand the proof of
Gerstenhaber
and Schack, we give our own detailed proof of this proposition.
\begin{proof}
Let $V$ be a graded $\mathbb{Q}$-vector space.
Let $\Lambda V$ be the free graded commutative algebra on $V$.
By~\cite[Proposition 22.7]{Felix-Halperin-Thomas:ratht} applied to
$V$ considered as a differentiel graded abelian Lie algebra,
the linear map $\lambda_V:\Lambda (sV),0\buildrel{\simeq}\over\hookrightarrow
\overline{B}(\Lambda V)$, mapping $v_1\wedge\dots\wedge v_n$
to the shuffle product $[v_1]\star \dots\star [v_n]$,
is an injective quasi-isomorphism of differential graded coaugmented
coalgebras (This is a consequence of the (graded) Koszul resolution).

Let $A$ be an augmented algebra. Denote by
$\overline{B}A$ the normalized reduced Bar construction.
Let $\tau_A:s^{-1}\overline{B}A\rightarrow A$ be the linear map
of degre $0$, mapping $s^{-1}[sa_1|\dots|sa_n]$ to $a_1$ if $n=1$
and to $0$ otherwise.
Then the unique morphism of graded algebras
$\Omega \overline{B}A\buildrel{\simeq}\over\rightarrow A$,
extending $\tau_A$ is a quasi-isomorphism~\cite[Proposition
2.14]{Felix-Halperin-Thomas:Adamce}.
Suppose that $V$ is concentrated in negative (lower) degre,
the both $\Lambda sV$ and $\overline{B}(\Lambda V)$
are concentrated in non positive degre.
Therefore, by~\cite[Remark 2.3]{Felix-Halperin-Thomas:Adamce},
the morphism of differential graded algebras
$$\Omega\lambda_V:\Omega(\Lambda (sV),0)\buildrel{\simeq}\over\rightarrow
\Omega\overline{B}(\Lambda V)$$ is a quasi-isomorphism.
Therefore, by composing, we obtain a quasi-isomorphism
of differential graded algebras $$\nu_{sV}:\Omega\Lambda
sV\buildrel{\simeq}\over\rightarrow
\Omega\overline{B}(\Lambda V)\buildrel{\simeq}\over\rightarrow
(\Lambda V,0).$$
(This is a particular case of~\cite[line above Lemma
8.3]{MenichiL:cohrfl}.) Note that if $V$ is finite dimensional,
$\nu_V$ is the dual of $\lambda {sV^\vee}$.

Since $L$ is ungraded, $V=s^{-1}L$ is concentrated in degre $-1$
and therefore, we have the quasi-isomorphism of differential graded
algebras
$\nu_L:\Omega\Lambda L\buildrel{\simeq}\over\rightarrow (\Lambda
s^{-1}L,0)$.
Poincar\'e-Birkoff-Witt gives an isomorphism of coalgebras
$PBW:UL\buildrel{\cong}\over\rightarrow \Lambda L$ which restricts to the
identity map on the primitives.
By definition, $\nu_L:\Omega\Lambda L\rightarrow \Lambda (s^{-1}L)$
extends the composite
$$s^{-1}\Lambda L\buildrel{s^{-1}\lambda_{s^{-1}L}}\over\longrightarrow s^{-1}\overline{B}\Lambda
s^{-1}L\buildrel{\tau_{\Lambda s^{-1}L}}\over\longrightarrow\Lambda s^{-1}L  
$$
which maps $s^{-1}(l_1\wedge\dots\l_n)$ to $s^{-1}l_1$ if $n=1$ and to
$0$ otherwise.
Therefore, we have the commutative diagram
\[\xymatrix
{
\Omega UL\ar[r]^{\Omega(PBW)}_\cong
&\Omega \Lambda L\ar[r]^{\nu_L}_\simeq
& \Lambda s^{-1}L\\
s^{-1}PUL\ar[r]\ar[u]
&s^{-1}P\Lambda L\ar[r]\ar[u]
&s^{-1}L\ar[u]
}\]
where the vertical arrows are the canonical inclusions
and the bottom horizontal maps are the identity maps.
Therefore the inverse of the algebra isomorphism
$$
H_*(\Omega UL)\buildrel{H_*(\Omega PBW)}\over\longrightarrow
H_*(\Omega \Lambda L)\buildrel{H_*(\nu_L)}\over\longrightarrow
\Lambda s^{-1}L
$$
coincides with $\varphi_{UL}$.
\end{proof}
Let $L$ be a Lie algebra. A {\it character} of $L$ is by definition~\cite[Example 5.5]{Crainic:cylcohhopfalg} a morphism of Lie algebras $\delta:L\rightarrow {\Bbbk}$.
Let $A$ be a connected Batalin-Vilkovisky algebra.
Then $B=A^1\rightarrow A^0={\Bbbk}$ is a character for the Lie algebra $A^1$.
Indeed for $a$, $b\in A^1$, since $\{1,b\}=0$ and $\{a,1\}=0$,
by~\cite[Proposition 1.2]{Getzler:BVAlg},
$$B\{a,b\}=\{Ba,b\}\pm \{a,Bb\}=0\pm 0=0.$$
This forgetful functor from connected Batalin-Vilkovisky algebras to Lie algebras equipped with a character has a left adjoint (Compare with~\cite[Freely generated $BV_n$ algebras in Section 4]{Gaudens-Menichi}):

Let $L$ be a Lie algebra equipped with a character
$\delta:L\rightarrow {\Bbbk}$. Then $\lambda.x:=\delta (x) \lambda$ for
$x\in L$ and $\lambda\in {\Bbbk}$ defines a right (and also left) Lie action of $L$
on ${\Bbbk}$.
The differential $d:\Lambda^n L\rightarrow \Lambda^{n+1} L$
of the Chevalley-Eilenberg complex $C_*(L,{\Bbbk})$ is given
by~(\cite[(10.1.3.1)]{LodayJ.:cych} with the opposite differential
or~\cite[7.7.1]{Weibel:inthomalg} for exactly the same differential)
\begin{multline*}
d(x_1\wedge\dots\wedge x_n)=\sum_{i=1}^n(-1)^{i-1}\delta(x_i)
x_1\wedge\dots\wedge\widehat{x_i}\wedge\dots\wedge x_n\\
+\sum_{1\leq i<j\leq n}(-1)^{i+j}\{x_i,x_j\}\wedge x_1\wedge\dots\wedge\widehat{x_i}\wedge\dots\wedge\widehat{x_j}\wedge\dots\wedge x_n.
\end{multline*}
Here the symbol $\widehat{\;}$ denotes omission.
A direct calculation shows that
\begin{multline*}
$$(-1)^p(d(x_1\wedge\dots\wedge x_p\wedge y_1\wedge\dots\wedge y_q)\\
-d(x_1\wedge\dots\wedge x_p) y_1\wedge\dots\wedge y_q\\
(-1)^p x_1\wedge\dots\wedge x_pd(y_1\wedge\dots\wedge y_q)
)
\end{multline*}
is the Schouten bracket on the Gerstenhaber algebra $\Lambda^* L$ defined
by equation~(\ref{Schouten bracket}).
Therefore the Gerstenhaber algebra  $\Lambda^* L$ equipped with the operator $d$
is a Batalin-Vilkovisky algebra.
By induction, one can check that in any Batalin-Vilkovisky algebra $A$,
for $x_1,\dots,x_n\in A$,
\begin{multline*}
B(x_1\dots x_n)=\sum_{i=1}^n(-1)^{\vert x_1\vert+\dots+\vert x_{i-1}\vert}
x_1\dots B(x_i)\dots x_n\\
+\sum_{1\leq i<j\leq n}(-1)^{\vert x_1\vert+\dots+\vert x_{j-1}\vert+\vert x_1\dots x_{i-1}\vert\vert x_i\vert}
\{x_i,x_j\} x_1\dots \widehat{x_i}\dots\widehat{x_j}\dots x_n.
\end{multline*}
It follows easily that
the inclusion of Lie algebras with character $$(L,\delta)\rightarrow
(\Lambda^1 L, d:\Lambda^1 L\rightarrow \Lambda^0 L={\Bbbk})$$
is universal.

Let $C$ be a Hopf algebra endowed with a modular pair in involution of the
form $(\delta,1)$. By~\cite[(2.19) or (2.22)]{Connes-Moscovici:symmetry},
the operator $$B:\text{Cotor}_C^1({\Bbbk},{\Bbbk})=P(C)\rightarrow
\text{Cotor}_C^0({\Bbbk},{\Bbbk})={\Bbbk}$$
coincides with the restriction of $\delta$,
$\delta_{|P(C)}:P(C)\rightarrow {\Bbbk}$.
Since $\delta$ is a character for the associative algebra $C$,
$\delta_{|P(C)}$ is a character for the Lie algebra of primitive elements $P(C)$.
By universal property, the morphism of Gerstenhaber algebra
$\varphi_C:(\Lambda^*P(C),d)\rightarrow \text{Cotor}_C^*({\Bbbk},{\Bbbk})$
is a morphism of Batalin-Vilkovisky algebras between the free Batalin-Vilkovisky
algebra generated by the Lie algebra with character $P(C)$ and the Batalin-Vilkovisky algebra recalled in Corollary~\ref{Cobar algebre de Hopf BV}.
As an immediate consequence of Proposition~\ref{iso gerstenhaber cotor UL},
we obtain the following theorem:
\begin{theorem}\label{calcul BV-algebre cotor algebre envellopante sur Q}
Let $L$ be a Lie algebra over the rationals $\mathbb{Q}$.
Let $\delta:L\rightarrow \mathbb{Q}$ be a character for $L$.
Extend $\delta$ to a character $\delta:UL\rightarrow \mathbb{Q}$ for $UL$.
The morphism of Batalin-Vilkovisky algebras 
$$\varphi_{UL}:(\Lambda^* L,d)\rightarrow \text{Cotor}_{UL}^*(\mathbb{Q},\mathbb{Q})$$
is an isomorphism.
\end{theorem}
In~\cite[Theorem 4.4]{Gaudens-Menichi}, G\'erald Gaudens and the author
showed that the rational homology $H_*(\Omega^2 X;\mathbb{Q})$
of the double loop space of a $2$-connected space  $X$ is isomorphic as
Batalin-Vilkovisky algebras to 
$\left(\Lambda (\pi_*(\Omega X)\otimes\mathbb{Q}),d\right)$,
the free Batalin-Vilkovisky algebra generated by the (graded) Lie algebra
$\pi_*(\Omega X)\otimes\mathbb{Q}$ equipped with the trivial character.
The graded version of Theorem~\ref{calcul BV-algebre cotor algebre envellopante sur Q}
shows that $\left(\Lambda (\pi_*(\Omega X)\otimes\mathbb{Q}),d\right)$ is isomorphic
as Batalin-Vilkovisky algebras to $\text{Cotor}_{H_*(\Omega X;\mathbb{Q})}^*(\mathbb{Q},\mathbb{Q})$.
So finally, we obtain an isomorphism of Batalin-Vilkovisky algebras
$$
H_*(\Omega^2 X;\mathbb{Q})\cong \text{Cotor}_{H_*(\Omega X;\mathbb{Q})}^*(\mathbb{Q},\mathbb{Q}).
$$
Of course, such isomorphism must be compared with our
Conjecture~\ref{conjecture iso homologie lacets double cotor}.
\bibliography{Bibliographie}
\bibliographystyle{amsplain}
\end{document}